\documentclass[reqno,10pt,A4paper]{amsart}

\usepackage[english]{babel}
\usepackage{amsfonts, amsmath, amssymb, amsthm, mathrsfs}

\usepackage{hyperref}
\usepackage{mathtools}  
\mathtoolsset{showonlyrefs}	

\usepackage{amsthm}
\newtheorem{definition}{Definition}
\newtheorem{lemma}[definition]{Lemma}
\newtheorem{proposition}[definition]{Proposition}
\newtheorem{theorem}[definition]{Theorem}
\theoremstyle{remark}

\newtheorem{remark}[definition]{Remark}

\usepackage{hyperref}

\newcommand{\C}{\mathbb{C}}
\newcommand{\R}{\mathbb{R}}
\newcommand{\N}{\mathbb{N}}

\renewcommand{\Re}{\operatorname{Re}}
\renewcommand{\Im}{\operatorname{Im}}

\newcommand{\ensemble}[4]{\{ #1 \}_{#2\le #3 \le #4}}
\newcommand{\setleft}[2]{ \left\{ \left. #1 \,\right|\, #2 \right\} }
\newcommand{\setright}[2]{ \left\{ #1 \,\left|\, #2 \right.\right\} }

\newcommand{\Rono}{\R^{1,n+1}}
\newcommand{\Hn}{\mathbb{H}^{n+1}}
\newcommand{\Sn}{\mathbb{S}^n}
\newcommand{\Un}{\mathbb{U}^{n+1}}

\newcommand{\SO}{\mathrm{SO}}

\newcommand{\End}{\operatorname{End}}
\newcommand{\T}{\mathrm{T}}

\newcommand\WF{\mathrm{WF}}

\newcommand\Dist{\mathcal{D}'}

\newcommand{\supp}{\mathrm{supp}}

\newcommand{\even}{\mathrm{even}}

\newcommand{\grad}{\operatorname{grad}}

\newcommand{\Lie}{\mathcal{L}}

\newcommand{\n}{\nabla}
\newcommand\tr{\operatorname{tr}}

\newcommand{\ve}{\varepsilon}

\newcommand{\rlnm}{\rho^{\lambda+\frac n2 - m}}
\newcommand{\drr}{\tfrac{d\rho}{\rho}}
\newcommand{\rdr}{{\rho\partial_\rho}}

\newcommand{\Sym}{\mathrm{Sym}}

\newcommand\sA{\mathscr{A}}
\newcommand\co[3]{{\{#1\to#2\}#3}}

\newcommand{\ip}{\operatorname{\lrcorner}}
\renewcommand{\sp}{\operatorname{\cdot}}

\newcommand{\trace}{\operatorname{\Lambda}}
\newcommand{\adjtrace}{\operatorname{L}}

\renewcommand{\d}{\operatorname{d}}
\renewcommand{\div}{\operatorname{\delta}}

\newcommand{\Riemann}{\operatorname{R}}

\usepackage{tensor}

\newcommand\twovector[2]{\left[\begin{array}{c} 	#1 \\ #2 \end{array}\right]}
\newcommand\twomatrix[4]{\left[\begin{array}{cc}	#1&#2 \\
										#3&#4 \end{array}\right]}

\newcommand{\nm}{\n_-}
\newcommand{\dm}{\d_-}
\newcommand{\divp}{\div_+}

\newcommand{\Deltap}{\Delta_+}

\newcommand{\Nm}{N^-}
\newcommand{\Np}{N^+}
\newcommand{\Npm}{N^\pm}

\newcommand{\V}[3]{V_{#1}^{#2}(#3)} 

\newcommand{\Bd}{\mathrm{Bd}} 

\newcommand{\Poisson}{\operatorname{\mathcal{P}}}

\newcommand{\Resolvent}[2]{\mathcal{R}_{#1}(#2)}
\newcommand{\ResolventHol}[2]{\mathcal{R}^{\mathrm{Hol}}_{#1}(#2)}
\newcommand{\Resonant}[3]{\mathrm{Res}_{#1}^{#2}(#3)}
\newcommand{\Residue}[2]{\mathrm{Res}_{#1}({#2})}
\newcommand{\Projector}[2]{\textstyle \prod_{#1}^{#2}}

\newcommand{\AVasyDown}{\operatorname{\mathcal{A}}}

\usepackage{leftidx} 

\newcommand{\Xone}{X}

\newcommand{\Xthree}{X_{e}}

\newcommand{\Mone}{M}
\newcommand{\Mthree}{M_{e}}

\newcommand{\gamb}{\eta}

\newcommand{\s}{s}

\newcommand{\sds}{\s\partial_\s}
\newcommand{\dss}{\tfrac{d\s}{\s}}

\renewcommand{\t}{t}
\newcommand{\RplusT}{\R^+_\t}

\newcommand{\Lsections}{L^2}

\newcommand{\E}{\mathcal{E}}
\newcommand{\F}{\mathcal{F}}

\newcommand{\DeltaAmb}{\operatorname{\square}}

\newcommand{\QVasyUp}{\operatorname{\mathbf{Q}}}
\newcommand{\QVasyDown}{\operatorname{\mathcal{Q}}}
\newcommand{\QVasyDOWNL}{\QVasyDown_\lambda}

\newcommand{\PVasyUp}{\operatorname{\mathbf{P}}}
\newcommand{\PVasyDown}{\operatorname{\mathcal{P}}}
\newcommand{\PVasyDOWNL}{\PVasyDown_\lambda}

\newcommand{\DVasyUp}{\mathbf{D}}
\newcommand{\DVasyDown}{\mathcal{D}}

\newcommand{\ZeroOrderVasyUp}{\mathbf{G}}
\newcommand{\ZeroOrderVasyDown}{\mathcal{G}}

\title[Ruelle and Quantum Resonances]
{Ruelle and Quantum Resonances for Open Hyperbolic Manifolds}
\author[C. Hadfield]{Charles Hadfield}
\address{
\'Ecole Normale Sup\'erieure, 45 rue d'Ulm,
75230 Paris cedex 05, France}
\email{charles.hadfield@ens.fr}

\usepackage{color}

\begin{document}


\begin{abstract}
We establish a direct classical-quantum correspondence on convex cocompact hyperbolic manifolds between the spectrums of the geodesic flow and the Laplacian acting on natural tensor bundles. This extends previous work detailing the correspondence for cocompact quotients.
\end{abstract}

\maketitle
\thispagestyle{empty}


\section{Introduction}

On a closed hyperbolic surface, Selberg's trace formula
\cite{selberg}
establishes a connection between eigenvalues of the Laplacian (on functions) and closed geodesics via the Selberg zeta function. In the convex cocompact setting this result is established by Patterson and Perry
\cite{patterson-perry}.
In this open setting the role which the eigenvalues of the Laplacian played is now played by quantum resonances for the Laplacian. That is, the poles of the meromorphic extension of the resolvent of the Lapacian.
Although these results indicate a correspondence between classical and quantum phenomena, it is somewhat indirect as it uses the closed geodesics to represent classic phenomena rather than treating directly the vector field which generates the geodesic flow on the unit tangent bundle.

The geodesic flow on the unit tangent bundle of a closed hyperbolic surface is an example of an Anosov flow. Considerable attention has been given to such flows recently using functional analytical techniques \cite{butterley-liverani} and microlocal methods \cite{faure-sjostrand} and has led to striking results including the meromorphic extension to the complex plane of the Ruelle zeta function of a $C^\infty$ Anosov flow on a compact manifold \cite{giulietti-liverani-pollicott, dyatlov-zworski}. The microlocal methods presented in \cite{faure-sjostrand} have been extended to the setting where the manifold need not be compact in \cite{dyatlov-g} in order to study Axiom A flows. An example of such a flow is the geodesic flow on the unit tangent bundle of a convex cocompact hyperbolic surface. The resolvent of said flow has a meromorphic extension to the complex plane whose poles define Ruelle resonances for this flow.

Returning to the classical-quantum correspondence of interest in this article, it is Dyatlov, Faure, and Guillarmou \cite{dfg} who establish a direct link between eigenvalues of the Laplacian on a closed hyperbolic surface and Ruelle resonances of the generator of the geodesic flow on the unit tangent bundle. This result had previously been announced by Faure and Tsujii \cite[Proposition 4.1]{faure-tsujii}. The extension to the convex cocompact setting, showing the link between quantum resonances for the Laplacian and Ruelle resonances for the generator of the geodesic flow, has recently been established by Guillarmou, Hilgert, and Weich \cite{ghw}.

The article \cite{dfg} studies not only surfaces, but rather cocompact quotients of hyperbolic space of any dimension. Interestingly, in this higher dimensional setting, the correspondence is no longer simply between Ruelle resonances and the spectrum of the Laplacian acting on functions, but rather the spectrums of the Laplacian acting on symmetric tensors (precisely, those tensors which are trace-free and divergence-free). The goal of this present article is to establish the classical-quantum correspondence in the convex cocompact setting for manifolds of dimension at least 3.

Let us now be more explicit and denote by  
$\Xone$ 
a convex cocompact quotient of hyperbolic space
$\Hn$
where
$n\ge 2$.

Consider first the classical phenomena. Denote by 
$A$
the generator of the geodesic flow 
(a tangent vector field on the unit tangent bundle
$S\Xone$).
The operator 
$A+\lambda$
is invertible as an operator on
$\Lsections$
sections whenever
$\Re\lambda>0$.
Let us introduce the following notation for its resolvent:
$\Resolvent{A,0}{\lambda}=(A+\lambda)^{-1}$.
By
\cite{dyatlov-g},
the resolvent admits a meromorphic extension
$\Resolvent{A,0}{\lambda}:C_c^\infty(S\Xone) \to \Dist(S\Xone)$
for
$\lambda\in\C$
whose poles are of finite rank.
These poles are called Ruelle resonances.
In fact the result of
\cite{dyatlov-g}
is very robust and can be used for more general objects than flows acting on functions. In particular, we note that the tangent bundle 
$\T\Xone$
over
$\Xone$
may be pulled back to a bundle over
$S\Xone$ 
which decomposes canonically into a line bundle spanned by 
$A$ 
and the perpendicular
$n$-dimensional
Euclidean
bundle denoted
$\E$.
Considering the flow as parallel transport, it is easy to extend the vector field
$A$ 
to a first-order differential operator on the tensor bundle 
$\E^*$
as well as on symmetric tensor products of
$\E^*$.
Again,
$A+\lambda$
is invertible as an operator on 
$\Lsections$
sections of
$\Sym^m\E^*$
whenever
$\Re\lambda>0$, 
and its resolvent admits a meromorphic extension
\begin{align}
\Resolvent{A,m}{\lambda} :C_c^\infty(S\Xone;\Sym^m\E^*) \to \Dist(S\Xone;\Sym^m\E^*)
\end{align}
for 
$\lambda\in\C$
whose poles are of finite rank. For a pole
$\lambda_0$,
the residue is a finite rank operator
$\Projector{A,m}{\lambda_0}$
whose image defines the set of generalised Ruelle resonant states of tensor order 
$m$.
These states are characterised by a precise support and wave-front condition detailed in 
Section~\ref{sec:ruelle}
as well as the fact that they are annihilited by some power of 
$A+\lambda_0$. 
In the cocompact setting, the poles are necessarily simple, hence such a state is annihilated immediately by
$A+\lambda_0$. 
The convex cocompact case may include non-simple poles and states are generalised in the sense that a power of 
$A+\lambda_0$
is required to annihilate the state.
Said power is called the Jordan order of the state.
The set of generalised Ruelle resonant states of tensor order $m$ associated with the pole $\lambda_0$ is denoted
\begin{align}
\Resonant{A,m}{}{\lambda_0}.
\end{align}

Consider second the quantum phenomena. The Levi-Civita connection of
$\Xone$
gives the positive rough Laplacian
$\n^*\n$.
On functions, the operator 
$\n^*\n-s(n-s)$
is invertible on 
$\Lsections$
sections whenever
$\Re\s \gg 1$
and the resolvent 
$\Resolvent{\Delta,0}{s} = ( \n^*\n-s(n-s) )^{-1}$
admits a meromorphic extension 
$\Resolvent{\Delta,0}{s} : C_c^\infty(\Xone) \to \rho^{s} C^\infty_\even(\overline \Xone)$
for
$s\in\C$ whose poles are of finite rank,
see
\cite{mazzeo-melrose, guillope-zworski:pb, g:duke}.
Here,
$\rho$
denotes any even boundary defining function for
$\Xone$
(detailed in 
Section~\ref{sec:quantum}).
This result is usually stated for the more general geometry of manifolds which are even asymptotically hyperbolic 
\cite[Definition 5.2]{g:duke}.
Analogous to the previous paragraph, the poles of the meromorphic extension define quantum resonances.
The correspondence discovered in 
\cite{dfg}
appeals to the eigenvalues of the Laplacian acting on symmetric tensors, thus a notion of quantum resonances for such a Laplacian is also required for the convex cocompact setting. This has recently been obtained in \cite{h:quantum} using Vasy's method
\cite{v:ml:inventiones, v:ml:functions, v:ml:forms, zworski:vm}.
Denote symmetric 
$m$-tensors
$\Sym^m\T^*\Xone$
which are trace-free by
$\Sym^m_0\T^*\Xone$.
The positive rough Laplacian acts on these tensors and the operator
$\n^*\n-s(n-s)-m$
is invertible on
$\Lsections$
sections whenever
$\Re\s \gg 1$.
Morevoer, when restricting further to those tensors which are divergence-free,
$\ker\div$,
the resolvent admits a meromorphic extension
\begin{align}
\Resolvent{\Delta,m}{s} 
: C_c^\infty(\Xone;\Sym^m_0 \T^*\Xone) \cap \ker \div 
\to
\rho^{s -m} C^\infty_\even(\overline \Xone;\Sym^m_0 \T^*\Xone) \cap \ker \div. 
\end{align}
for
$s\in\C$
whose poles are of finite rank.
For a pole 
$s_0$,
the residue is a finite rank operator
$\Projector{\Delta,m}{s_0}$
whose image defines the set of generalised quantum resonant states of tensor order
$m$
\begin{align}
\Resonant{\Delta,m}{}{s_0}.
\end{align} 
Unlike states associated with Ruelle resonances, a characterisation of states associated with quantum resonances was not previously available, this is remedied in 
Lemma~\ref{lem:AssOfQRes} 
detailing the asymptotic structure of such states.
The proof of this lemma is inspired by 
\cite[Proposition 4.1]{ghw}
but it relies heavily on various operators constructed in 
\cite{h:quantum}
and it does not seem to follow in any direct manner from the mere existence of the meromorphic extension for the resolvent of the Laplacian.

With the two notions of quantum and classical resonant states introduced, we may announce
\begin{theorem}
\label{thm:CQCorrespondence}
Let 
$\Xone = \Gamma \backslash \Hn$ 
be a smooth oriented convex cocompact hyperbolic manifold with 
$n\ge2$, 
and 
$ \lambda_0 \in  \C \backslash ( -\tfrac n2 - \tfrac 12 \N_0 ) $.
There exists a vector space linear isomorphism between generalised Ruelle resonant states
\begin{align}
\Resonant{A,0}{}{\lambda_0}
\end{align}
and the following space of generalised quantum resonant states
\begin{align}
\bigoplus_{m\in\N_0}
\bigoplus_{k=0}^{\lfloor \frac m2 \rfloor}
\Resonant{\Delta,m-2k}{}{\lambda_0+m+n}.
\end{align}
\end{theorem}
In the following paragraph we sketch one direction of this correspondence. We mention here that a key ingredient is the Poisson operator used to identify Ruelle resonant states and quantum resonant states via equivariant distributions
on 
$\Sn=\partial\overline\Hn$.
This operator is very finely studied in
\cite{dfg}
and shown to be an isomorphism outside of the exceptional set
$-\tfrac n2 - \tfrac 12 \N_0$.
It would certainly be interesting to study more closely this exceptional set likely leading to topological and conformal considerations. Indeed in the setting of convex cocompact surfaces 
\cite{ghw}, 
poles in the exceptional set are studied and are related to the topology of the surfaces.

Let us briefly sketch one direction of the isomorphism announced in
Theorem~\ref{thm:CQCorrespondence}.
Consider a Ruelle resonance
$\lambda_0\in\C$ 
which, for simplicity, we will assume is a simple pole.
Associated with
$\lambda_0$,
consider a Ruelle resonant state
\begin{align}
u\in\Resonant{A,0}{}{\lambda_0}.
\end{align}
That is,
$u\in \Dist(S\Xone)$
solves
$(A+\lambda_0)u=0$
subject to a wave-front condition detailed in 
Section~\ref{sec:ruelle}.
A non-trivial idea contained in
\cite{dfg}
is the construction of horosphere operators that generalise the horocycle vector fields present for hyperbolic surfaces. Specifically, recalling the $n$-dimensional bundle
$\E$,
there exists a differential operator
\begin{align}
\d_- : C^\infty(S\Xone;\Sym^m \E^*) \to C^\infty(S\Xone ; \Sym^{m+1}\E^*)
\end{align}
which may be morally thought of as a symmetric differential along the negative horospheres.
Moreover, this operator enjoys the commutation relation
\begin{align}
[A,\d_-]=-\d_-.
\end{align}
As (tensor valued) Ruelle resonances are also restricted to 
$\Re\lambda\le0$,
this commutation relation implies the existence of
$m\in\N_0$
such that
$v:=(\d_-)^m u \neq 0$
and
$\d_-v=0$.
Moreover,
$(A+\lambda_0+m)v=0$.
As the vector bundle $\E^*$ carries a natural metric, we have a notion of a trace operator,
$\trace$ 
and its adjoint 
$\adjtrace$
acting on 
$\Sym^m\E^*$. 
Denoting the bundle of trace-free symmetric tensors of rank $m$ by
$\Sym^m_0\E^*$, 
we decompose 
$v$
into trace-free components
\begin{align}
v = \sum_{k=0}^{\lfloor \frac m2 \rfloor} \adjtrace^{k} v^{(m-2k)},
\qquad
v^{(m-2k)} \in \Dist(S\Xone ; \Sym^{m-2k}_0 \E^*) \cap \ker (A+\lambda_0+m).
\end{align}
Integrating over the fibres of 
$S\Xone\to \Xone$
allows
$v^{(m-2k)}$ 
to be pushed to a symmetric $(m-2k)$-tensor on
$\Xone$
\begin{align}
\varphi^{(m-2k)} := \pi_{0*} v^{(m-2k)} \in C^\infty( \Xone ; \Sym^{m-2k} \T^*\Xone).
\end{align}
and the properties of the Poisson transform
imply that
\begin{align}
\varphi^{(m-2k)} \in \ker (\n^*\n + (\lambda_0+m)(n+\lambda_0+m)-(m-2k))
\end{align}
Lemma~\ref{lem:AssOfQRes} 
gives a classification of generalised quantum resonant states from which we conclude
that 
$\varphi^{(m-2k)}$ 
is indeed a generalised quantum resonant state associated with the resonance
$\lambda_0+m+n$.
(In fact it is a true quantum resonant state as it is immediately killed by
$\n^*\n + (\lambda_0+m)(n+\lambda_0+m)-(m-2k)$
rather than by a power thereof.)
Stated differently,
\begin{align}
\varphi^{(m-2k)} \in  \Resonant{\Delta,m-2k}{}{\lambda_0+m+n}.
\end{align}

Two aspects of the argument render the isomorphism considerably labour intensive. First, one needs to deal with inverting the horosphere operators, however for this, we may appeal to calculations from
\cite[Section 4]{dfg}
appealing to a polynomial structure present in the proof.
Second, one needs to consider the possibility that the Ruelle resonance is not a simple pole, but rather, there may exist generalised Ruelle resonant states.

This article is structured as follows.
Section~\ref{sec:prelim}
recalls from 
\cite{hms}
conventions for symmetric tensors.
Section~\ref{sec:hspace}
recalls numerous objects on hyperbolic space which are present in
\cite{dfg}
and which also descend to objects on convex cocompact quotients.
Section~\ref{sec:ruelle}
examines Ruelle resonances in the current setting.
It provides a key result from 
\cite{dyatlov-g}
which characterises Ruelle resonances (and generalised resonant states).
It also recalls the band structure of Ruelle resonances due to the Lie algebra commutation relations.
The section finishes with a restatement of
\cite[Lemma 4.2]{dfg}
emphasising a polynomial structure and which allows the inversion result for horosphere operators to be used in the presence of Jordan blocks.
Section~\ref{sec:quantum}
recalls the construction of various operators \`a la Vasy used to obtain the meromorphic extension of the resolvent of the Laplacian on symmetric tensors in 
\cite{h:quantum}.
It then characterises generalised quantum resonant states via their asymptotic structure.
Section~\ref{sec:poisson}
introduces boundary distributions which are the intermediary objects between quantum and classical resonant states and shows that the Poisson operator remains an isomorphism in the convex cocompact setting.
To finish, Section~\ref{sec:proof}
collects the results provided in the previous sections to succinctly prove 
Theorem~\ref{thm:CQCorrespondence}.


\section{Symmetric Tensors}\label{sec:prelim}

\subsection{A single fibre}
Let 
$E$ 
be a vector space of dimension 
$n$ 
equipped with an inner product 
$g$.
Use
$g$
to identify 
$E$ 
with its dual space.
Let
$\ensemble{e_i}{1}{i}{n}$ 
be an orthonormal basis.
We denote by 
$\Sym^m E$ 
the 
$m$-fold 
symmetric tensor product of 
$E$. 
Elements are symmetrised tensor products
\begin{align}
u_1\sp \ldots \sp u_m 
:= 
\sum_{\sigma \in \Pi_m} u_{\sigma(1)}\otimes \ldots\otimes u_{\sigma(m)},
\qquad
u_i\in E
\end{align}
where 
$\Pi_m$ 
is the permutation group of 
$\{1,\dots,m\}$. 
By linearity, this extends the operation 
$\sp$ 
to a map from 
$\Sym^mE\times \Sym^{m'}E$.
Some notation for finite sequences is required for calculations with symmetric tensors, and which is used in
Lemma~\ref{lem:EquivarianceOfSections}.
Denote by 
$\sA^m$ 
the space of all sequences 
$K=k_1\dots k_m$ 
with 
$1\leq k_r\leq n$. 
We write 
$\co{k_r}{j}{K}$ 
for the result of replacing the 
$r$\textsuperscript{th} 
element of 
$K$ 
by 
$j$. 
We set
\begin{align}
e_K 
:= 
e_{k_1}\sp \ldots \sp e_{k_m} \in\Sym^m E,
\qquad
K=k_1\dots k_m \in \sA^m.
\end{align}

The inner product induces an inner product on 
$\Sym^mE$,
also denoted by
$g$,
defined by
\begin{align}
g( u_1\sp\ldots\sp u_m, v_1\sp\ldots\sp v_m )
:= 
\sum_{\sigma\in\Pi_m} g(u_1,v_{\sigma(1)})\ldots g(u_m,v_{\sigma(m)}),
\qquad 
u_{i},v_{i}\in E.
\end{align}
For 
$u\in E$, 
the metric adjoint of the linear map 
$u\sp : \Sym^mE\to\Sym^{m+1}E$ 
is the contraction 
$u \ip : \Sym^{m+1}E\to\Sym^mE$.
Contraction and multiplication with the metric 
$g$ 
define two additional linear maps $\trace$ and $\adjtrace$,
\begin{align}
\trace : \left\{	 \begin{array}{rcl}
			\Sym^m E & \to & \Sym^{m-2} E \\
			u & \mapsto & \sum_{i=1}^n   e_i \ip e_i \ip u
		\end{array}
		\right.
&&
\adjtrace : \left\{	 \begin{array}{rcl}
			\Sym^m E & \to & \Sym^{m+2} E \\
			u & \mapsto & \sum_{i=1}^n e_i \sp e_i \sp u
		\end{array}
		\right.
\end{align}
which are adjoint to each other. 
As the notation is motivated by standard notation from complex geometry, we will refer to these two operators as Lefschetz-type operators. 
Denote by
\begin{align}
\Sym^m_0 E : = \ker \left(\trace : \Sym^m E \to \Sym^{m-2} E\right)
\end{align}
the space of trace-free symmetric tensors of degree
$m$.

\subsection{Vector bundles}

The previous constructions may be performed using a Riemannian manifold's tangent bundle.
In view of
Section~\ref{sec:quantum},
consider
a Riemannian manifold 
$(\Xone,g)$ 
of dimension 
$n+1$
with Levi-Civita connection
$\n$. The rough Laplacian on
$\Sym^m \T\Xone$ 
is denoted
$\n^*\n$
(and equal to 
$-\tr_g \circ \n\circ \n$).

Let
$\ensemble{e_i}{0}{i}{n}$
be a local orthonormal frame.
The symmetrisation of the covariant derivative, called the symmetric differential, is
\begin{align}
\d : \left\{	 \begin{array}{rcl}
			C^\infty(\Xone ; \Sym^m \T \Xone) & \to & C^\infty(\Xone ; \Sym^{m+1} \T \Xone) \\
			u & \mapsto & \sum_{i=0}^n e_i\sp \n_{e_i} u
		\end{array}
		\right.
\end{align}
and its formal adjoint, called the divergence, is
\begin{align}
\div : \left\{	 \begin{array}{rcl}
			C^\infty(\Xone ; \Sym^{m+1} \T \Xone) & \to & C^\infty(\Xone ; \Sym^m \T \Xone) \\
			u & \mapsto & -\sum_{i=0}^n e_i \ip \n_{e_i} u
		\end{array}
		\right.
\end{align}
The two first-order operators behave nicely with the associated Lefschetz-type operators
$\adjtrace$ 
and 
$\trace$ 
giving the following commutation relations 
\cite[Equation 8]{hms}:
\begin{align}\label{eqn:HMScommutations}
[\trace,\div]=0=[\adjtrace,\d],
\quad
[\trace,\d]=-2\div,
\quad
[\adjtrace,\div]=2\d.
\end{align}


\section{Hyperbolic Space}\label{sec:hspace}

We recall the hyperbolic space as a submanifold of Minkowski space, introducing structures present in this constant curvature case. Enumerate the canonical basis of 
$\Rono$ 
by 
$e_0, \dots e_{n+1}$ 
and provide 
$\Rono$ 
with the indefinite inner product
$\langle x, y \rangle := - x_0 y_0 + \sum_{i=1}^{n+1} x_i y_i.$
Hyperbolic space, 
$\Hn$, 
a submanifold of 
$\Rono$, 
is
$
\Hn
:=
\setleft{x \in \Rono}{\langle x, x \rangle = -1 , x_0 >0 }
$ 
supplied with the Riemannian metric, 
$g$, 
induced from restriction of 
$\langle \cdot, \cdot \rangle$,
and Levi-Civita connection
$\n$. 
The unit tangent bundle is
$
S\Hn 
:= 
\setright{	
	(x,\xi) 	
}{	
	x\in\Hn,
	\xi\in\Rono, 
	\langle \xi, \xi \rangle = 1,
	\langle x, \xi \rangle=0
}.
$ 
Define the projection
$
\pi_S : S\Hn \to \Hn : (x,\xi) \mapsto x
$ 
and denote by
\begin{align}
\varphi_t : \left\{	 \begin{array}{rcl}
			S\Hn & \to & S\Hn \\
			(x,\xi) & \mapsto & (x \cosh t + \xi \sinh t , x \sinh t + \xi \cosh t) 
		\end{array}
		\right.
\end{align}
the geodesic flow for 
$t\in\R$ 
with generator denoted 
$A$.
That is,
$
A_{(x,\xi)} := (\xi, x).
$ 
The tangent space
$\T S\Hn$
at
$(x,\xi)$
may be written
\begin{align}
\T_{(x,\xi)} S\Hn
:=
\setleft{
	(v_x, v_\xi) \in (\Rono)^2
}{
	\langle x ,v_x \rangle =
	\langle \xi, v_\xi \rangle =
	\langle x, v_\xi \rangle + \langle \xi , v_x \rangle =
	0
}.
\end{align}
It has a smooth decomposition, invariant under $\varphi_{t*}$,
$
\T  S\Hn
=
E^n \oplus E^s \oplus E^u
$ 
where
\begin{align}
{E^n}_{(x,\xi)}
&:=	\setleft{
	(v_x, v_\xi)
	}{
	(v_x, v_\xi) \in \mathrm{span} \{ (\xi, x) \}
	},
\\
{E^s}_{(x,\xi)}
&:=	\setleft{
	(v, -v)
	}{
	\langle x, v \rangle =
	\langle \xi, v \rangle =
	0
	},
\\	
{E^u}_{(x,\xi)}
&:=	\setleft{
	(v, v)
	}{
	\langle x, v \rangle =
	\langle \xi, v \rangle =
	0
	}	
\end{align}
are respectively called the neutral, stable, unstable bundles 
(of $\varphi_{t*}$).
(The latter two also being tangent to the positive and negative horospheres.)
The dual space has a similar decomposition
$
\T^* S\Hn
=
E^{*n} \oplus E^{*s} \oplus E^{*u}
$ 
where 
$E^{*n}, E^{*s}, E^{*u}$
are respectively the dual spaces to
$E^n, E^u, E^s$. 
(They are the neutral, stable, unstable bundles of $\varphi_{-t}^*$.)
Explicitly
\begin{align}
{E^{*n}}_{(x,\xi)}
&:=	\setleft{
	(v_x, v_\xi)
	}{
	(v_x, v_\xi) \in \mathrm{span}\{ (\xi, x) \}
	},
\\
{E^{*s}}_{(x,\xi)}
&:=	\setleft{
	(v, v)
	}{
	\langle x, v \rangle =
	\langle \xi, v \rangle =
	0
	},
\\	
{E^{*u}}_{(x,\xi)}
&:=	\setleft{
	(v, -v)
	}{
	\langle x, v \rangle =
	\langle \xi, v \rangle =
	0
	}	
\end{align}
so we have canonical identifications
$E^{n*} \simeq E^n \simeq \mathrm{span} \{ A \}$,
and
$E^{*s} \simeq E^u$,
and
$E^{*u} \simeq E^s$.
Consider the pullback bundle
$
\pi_S^* \T \Hn \to S\Hn
$
equipped with the pullback metric, also denoted
$g$.
Define
\begin{align}
\E :=	\setleft{
	(x,\xi,v)\in S\Hn \times \T_x\Hn
	}{
	\langle \xi, v \rangle =0
	}
\end{align}
and
$
\F :=	\setleft{
	(x,\xi,v)\in S\Hn \times \T_x\Hn
	}{
	v \in \mathrm{span} \{ \xi \}
	}
$ 
so that
$
\pi_S ^* \T \Hn = \E \oplus \F.
$ 
Appealing to 
Section~\ref{sec:prelim},
we obtain bundles
we obtain the bundles
$\Sym^m\E^*$
above
$S\Hn$
and Lefschetz-type operators 
$\adjtrace,\trace$.

There are canonical identifications from $\E$ to both $E^s$ and $E^u$, which we denote by $\theta_\pm$:
\begin{align}
\begin{array}{ l }
\theta_+ : \E \to E^s: \\
\theta_- : \E \to E^u:
\end{array}
\qquad
{\theta_\pm}_{(x,\xi)} (v) : = (v , \mp v).
\end{align}

\subsection{Isometry group}
The group 
$\SO(1,n+1)$ 
of linear transformations of 
$\Rono$ 
preserving 
$\langle\cdot,\cdot\rangle$ 
provides the group 
\begin{align}
G:=\SO_0(1,n+1),
\end{align}
the connected component in 
$\SO(1,n+1)$ 
of the identity. 
Denote by 
$\gamma\cdot x$,
multiplication of 
$x\in\Rono$
by
$\gamma\in G$.
Denote by 
$E_{ij}$ 
is the elementary matrix such that 
$E_{ij} e_k = e_i\delta_{jk} $
and define the following matrices
\begin{align}
R_{ij} :=E_{ij} - E_{ji},
\qquad
P_k :=E_{0k} + E_{k0}
\end{align}
for
$1\le i,j,k\le n+1$. The Lie algebra,
$\mathfrak{g}$,
of
$G$
is then identified with
$
\mathfrak{g} 
= \mathfrak{k} + \mathfrak{p}
$ 
where
\begin{align}
\mathfrak{k} := \mathrm{span} 
						\ensemble{ R_{ij} }{1}{i,j}{n+1}	
			\simeq \mathfrak{so}_{n+1}, 
\qquad
\mathfrak{p} := \mathrm{span}  
						\ensemble{ P_k }{1}{k}{n+1}.
\end{align}
An alternative description of
$\mathfrak{g}$
may be obtained by defining
\begin{align}
A:=P_{n+1},
\qquad
\Npm_k :=P_k\pm R_{n+1, k}
\end{align}
for 
$1\le k \le n$. Then
$
\mathfrak{g}
= \mathfrak{m} + \mathfrak{a} + \mathfrak{n}_+ + \mathfrak{n}_-
$ 
where
$\mathfrak{a}:= \mathrm{span}\{ A \}$
and
\begin{align}
\mathfrak{m} := \mathrm{span}
						\ensemble{ R_{ij} }{1}{i,j}{n}	
			\simeq \mathfrak{so}_{n}, 
\qquad
\mathfrak{n}_\pm := \mathrm{span}  
						\ensemble{ \Npm_k }{1}{k}{n}.
\end{align}
The matrices introduced enjoy the following commutator relations, for 
$1\le i,j \le n$
\begin{align}\label{eq:CommutationRelationsAN}
[A , \Npm_i] = \pm \Npm _i, 
\qquad
[\Npm_i , \Npm_j] = 0,
\qquad
[\Np_i , \Nm_j ] = 2 A\delta_{ij} + 2 R_{ij},
\end{align}
while
\begin{align}
[R_{ij},A]=0, 
\qquad
[R_{ij}, \Npm_k ] = \Npm_{i} \delta_{jk}  -  \Npm_j \delta_{ik}.
\end{align}
\begin{remark}
If we define
$\mathfrak{a}^\perp:= \mathfrak{p} / \mathfrak{a}$ 
whence
$\mathfrak{a}^\perp \simeq \ensemble{P_k}{1}{k}{n}$
then we may obtain identifications
\begin{align}
\theta_\pm : \mathfrak{a}^\perp \to \mathfrak{n}_\pm : P_k \mapsto \Npm_k.
\end{align}
\end{remark}
Elements of the Lie algebra 
$\mathfrak{g}$ 
are identified with left invariant vector fields on 
$G$. 
The Lie algebras 
$\mathfrak{k}, \mathfrak{m}$ 
give Lie groups 
$K$, $M$ 
considered subgroups of 
$G$. 
Now 
$G$ 
acts transitively on both
$\Hn$ 
and 
$S\Hn$ 
and the respective isotropy groups, for
$e_0\in\Hn$
and
$(e_0,e_{n+1})\in S\Hn$,
are precisely
$K$ 
and 
$M$.
Define projections
\begin{align}
\pi_K : G \to \Hn &: \gamma \mapsto \gamma \cdot e_0, \\
\pi_M : G \to S\Hn &: \gamma \mapsto (\gamma \cdot e_0 , \gamma \cdot e_{n+1} ).
\end{align}
As 
$A$ 
commutes with 
$M$, 
it descends to a vector field on 
$S\Hn$ 
via 
$\pi_{M*}$. 
It agrees with the generator of the geodesic flow justifying the notation. Similarly, the spans of 
$\ensemble{\Np_k}{1}{k}{n}$ 
and 
$\ensemble{\Nm_k}{1}{k}{n}$ 
are each stable under commutation with 
$M$ 
and via 
$\pi_{M*}$ 
are respectively identified with the stable and unstable subbundles 
$E^s, E^u$.

\subsection{Equivariant sections}

It is clear that distributions on $S\Hn$ may be considered  as distributions on $G$ which are annihilated by $M$. We denote such distributions
\begin{align}
\Dist(G)/\mathfrak{m} :=	\setleft{
					u \in \Dist(G)
					}{
					R_{ij}u=0,
					1\le i,j\le n
					}.
\end{align}
This is true for more general sections, in particular we have
\begin{lemma}\label{lem:EquivarianceOfSections}
Sections 
$\Dist(S\Hn ; \Sym^m \E^*)$
are equivalent to equivariant sections
\begin{align}
\Dist(G ; \Sym^m \R^n)/\mathfrak{m}
:=
\setleft{
		\sum_{K\in\sA^m} u_K e_K 
	}{
		R_{ij} u_K = \sum_{\ell=1}^k	\left(  
								u_\co{k_\ell}{i}{K}\delta_{jk_\ell}  - u_\co{k_\ell}{j}{K} \delta_{ik_\ell}
								\right)
			, 1\le i,j\le n
	}.
\end{align}
\end{lemma}
\begin{proof}
It suffices to consider the case 
$m=1$. 
Demanding that 
$u=\sum_{k=1}^n u_k e_k$ 
corresponds to a section of 
$\E^*$ 
requires precisely that
\begin{align}
0 = R_{ij} u 
 = \sum_{k=1}^n (R_{ij} u_k) e_k + u_k (R_{ij} e_k) 
 = \sum_{k=1}^n (R_{ij} u_k) e_k + u_k ( e_i \delta_{jk} - e_j\delta_{ik} )
\end{align} 
for
$1\le i,j\le n$.
Applying
$e_k\ip$ 
to this equation recovers
$
R_{ij} u_k = u_i \delta_{jk} - u_j \delta_{ik}. 
$
\end{proof}
A similar statement may be made for other (not necessarily symmetric) tensor bundles of
$\E$.

\subsection{Differential operators on $\E$}

We introduce several operators on (sections of tensor bundles of)
$\E$.
As
$\E$
may be viewed as a subbundle of
$\Rono$
above
$S\Hn$,
let
$\n^{\mathrm{flat}}$
denote the induced connection (upon projection onto $\E$ of the flat connection on $\Rono$). Now
\begin{align}
\n^{\mathrm{flat}} : \Dist(S\Hn ; \E^*) \to \Dist(S\Hn ; \T^*S\Hn\otimes \E^*)
\end{align}
however if we restrict to differentiating in either only the stable or only the unstable bundles 
$E^s, E^u$,
via composition with
$\theta_\pm$,
we obtain horosphere operators
$
\n_\pm := \n^{\mathrm{flat}}_{\theta_\pm}
$
and in general we obtain
\begin{align}
\n_\pm : \Dist(S\Hn ; \otimes^m\E^*) \to \Dist(S\Hn ; \otimes^{m+1}\E^*).
\end{align}
Symmetrising this operator we get the (positive and negative) horosphere symmetric derivatives and their divergences
\begin{align}
\d_\pm : \Dist(S\Hn ; \Sym^m\E^*) \to \Dist(S\Hn ; \Sym^{m+1}\E^*), \\
\div_\pm : \Dist(S\Hn ; \Sym^{m+1}\E^*) \to \Dist(S\Hn ; \Sym^m\E^*),
\end{align}
as well as the horophere Laplacians 
$\Delta_\pm := [\div_\pm, \d_\pm]$.

Considering these operators acting on equivariant sections of the corresponding vector bundles we have
\begin{align}
\n_\pm = \sum_{i=1}^n e_k \otimes \Lie_{ \Npm_k }
:
\Dist(G ; \otimes^m \R^n ) / \mathfrak{m}
\to
\Dist(G ; \otimes^{m+1} \R^n ) / \mathfrak{m}
\end{align}
where $\Lie$ is the Lie derivative.
(The appearance of merely the Lie derivative is because 
$\n_\pm$ 
uses 
$\n^{\mathrm{flat}}$ 
and 
$\Npm_i e_j = -(e_0+e_{n+1})\delta_{ij} \not\in \R^n$ 
for 
$1\le i,j\le n$.)
Similarly
\begin{align}
\d_\pm = \sum_{k=1}^n e_k \sp \Lie_{\Npm_k}, 
\qquad
\div_\pm = -\sum_{k=1}^n e_k \ip \Lie_{\Npm_k},
\qquad
\Delta_\pm = -\sum_{k=1}^n \Lie_{\Npm_k} \Lie_{\Npm_k}
\end{align}
on 
$
\Dist(G ; \Sym^m \R^n ) / \mathfrak{m}.
$

Continuing to consider equivariant sections we note that
$\Lie_A$
acts as a first order differential operator 
$\Dist(G ; \Sym^m \R^n ) / \mathfrak{m}$
due to the commutator relations
($A$ commutes with $M$).
As
$Ae_i = 0$,
there will be no ambiguity in denoting this operator simply
$A$.
From the perspective of sections directly on $\Sym^m\E$ we have
\begin{align}
A : = (\pi_S^*\n)_{A} : \Dist( S\Hn ; \Sym^m \E^*) \to \Dist ( S\Hn ; \Sym^m \E^*)
\end{align}
since 
$\pi_{S*}A=\xi$ 
at 
$(x,\xi)\in S\Hn$.

There are numerous useful relations between these operators. 
On 
$\Dist(S\Hn)$ 
the operators 
$(\n_\pm)^m$
and
$(\d_\pm)^m$
agree since
$[\Npm_i,\Npm_j]=0$.
As in 
Section~\ref{sec:prelim}, these operators have the same computation relations as given in
\eqref{eqn:HMScommutations}.
Moreover, due to the first commutation relation presented in 
\eqref{eq:CommutationRelationsAN},
these operators have simple commutation relations with
$A$
\begin{align}
[A , \d_\pm] = \pm \d_\pm,
\qquad 
[A , \div_\pm] = \pm \div_\pm,
\qquad 
[A , \Delta_\pm] = \pm 2\Delta_\pm.
\end{align}

\subsection{Several operators on hyperbolic space}

The metric on 
$\T\Hn$
allows the standard construction of the rough Laplacian
\begin{align}
\n^*\n : C^\infty(\Hn ; \Sym^m \T^*\Hn) \to C^\infty(\Hn ; \Sym^m \T^*\Hn).
\end{align}
Another common Laplacian on symmetric tensors is the Lichnerowicz Laplacian 
\cite{hms}.
For a general Riemannian manifold, the Lichnerowicz Laplacian is given by
$\n^*\n+q(\Riemann)$
where 
$q(\Riemann)$ 
is a curvature correction of zeroth order. On
$\Hn$,
the curvature operator takes the constant value
$q(\Riemann)=-m(n+m-1)$.
The divergence is
\begin{align}
\div : C^\infty(\Hn ; \Sym^m \T^*\Hn) \to C^\infty(\Hn ; \Sym^{m-1} \T^*\Hn)
\end{align}
and we continue to use the notation
$\adjtrace,\trace$
for the Lefschetz-type operators
associated with 
$\Sym^m \T^*\Hn$.

\subsection{Conformal boundary}

Hyperbolic space is projectively compact, and we identify the boundary of its compactification with the forward light cone
$
	\setright{
	(t, ty)
	}{
	t\in\R^+,
	y\in\Sn
	}
	\subset
	\Rono.
$ 
Now
$x\pm\xi$
belongs to this light cone for
$(x,\xi)\in S\Hn$
and this defines maps
\begin{align}
\Phi_\pm : S\Hn \to \R^+,
\qquad
B_\pm : S\Hn \to \Sn,
\end{align}
by declaring
$
x\pm \xi = \Phi_\pm (x,\xi) (1, B_\pm(x,\xi)).
$ 
The Poisson kernel is
\begin{align}
P : \left\{	 \begin{array}{rcl}
			\Hn \times \Sn & \to & \R^+ \\
			(x,y) & \mapsto & - \langle x , e_0 + y \rangle^{-1}
		\end{array}
		\right.
\end{align}
which permits the definition of
\begin{align}
\xi_\pm  : \left\{	 \begin{array}{rcl}
			\Hn \times \Sn & \to & S\Hn \\
			(x,y) & \mapsto & (x , \mp x \pm P(x,y) ( e_0 + y) )
		\end{array}
		\right.
\end{align}
This gives an inverse to 
$B_\pm(x,\cdot)$
in the sense that
$B_\pm (x , \xi_\pm (x,\nu) ) = \nu$
(implying that
$B_\pm$
is a submersion).
Moreover,
$
\Phi_\pm (x , \xi_\pm (x,y) ) = P(x,y)
$ 
The isometry group
$G$
acts on conformal infinity. There are maps
\begin{align}
T : G\times \Sn \to \R^+,
\qquad
U : G\times \Sn \to \Sn,
\end{align}
defined by
$
\gamma\cdot(1,y)
=
T_\gamma(y)(1,U_\gamma(y)).
$ 
Useful formulae are
\begin{align}
A \circ \Phi_\pm = \pm \Phi_\pm,
\qquad
 \Npm_k ( \Phi_\pm \circ \pi_M) = 0,
\qquad
B_\pm = \lim_{t\to\pm\infty} \pi_S \circ \varphi_t,
\end{align}
and
\begin{align}\label{eq:GEquivarianceOfTandU}
B_\pm (\gamma \cdot (x,\xi) ) = U_\gamma ( B_\pm (x,\xi)),
\qquad
\Phi_\pm ( \gamma \cdot (x,\xi)) = T_\gamma ( B_\pm (x,\xi)) \Phi_\pm (x,\xi).
\end{align}
We introduce the map
\begin{align}
\tau_\pm : \left\{	 \begin{array}{rcl}
			\E_{(x,\xi)} & \to & \T_{y := B_\pm(x,\xi)}\Sn \\
			v & \mapsto & v + \langle v, e_0 \rangle e_0 - \langle v, y\rangle y
		\end{array}
		\right.
\end{align}
which isometrically identifies 
$\E_{(x,\xi)}$ 
with
$\T_{B_\pm(x,\xi)} \Sn$.
It has an inverse
\begin{align}
{\tau_\pm}^{-1} : \left\{	 \begin{array}{rcl}
			\T_{B^\pm(x,\xi)}\Sn & \to & \E_{(x,\xi)} \\
			\zeta & \mapsto & \zeta + \langle \zeta , x \rangle (x\pm \xi)
		\end{array}
		\right.
\end{align}
and the adjoint of
$\tau_\pm$
is denoted
${\tau_\pm}^*$.
Restricting our attention to $\tau_-$ we note the following equivariance under $G$,
\cite[Equation 3.33]{dfg}
\begin{align}\label{eq:GEquivarianceOfTauMinus}
\left(		
	{\tau_-}_{\gamma\cdot (x,\xi)}	
\right)^{-1} 
\left( 	
	{U_{\gamma *} }_{| B_-(x,\xi)}	
	(\zeta)	
\right)
=
\frac{1}{T_\gamma ( B_-(x,\xi) ) }	
\,	
\gamma 
\cdot 
\left(  
	\left(
		{\tau_-}_{(x,\xi)} 
	\right)^{-1}
	(\zeta) 
\right)
\end{align}
for 
$\zeta\in \T_{B^\pm(x,\xi)} \Sn$.
The identification offered by $\tau_-$ permits a second important identification of distributions in the kernel of both 
$A$ 
and 
$\nm$ 
with boundary distributions. Define the operator 
\begin{align}
\mathcal{Q}_- : \left\{	 \begin{array}{rcl}
			\Dist(\Sn ; \otimes^m \T^* \Sn) & \to & \Dist(S\Hn ; \otimes^m \E^*) \\
			\omega & \mapsto & (\otimes^m (\tau_-^*) ). \omega \circ B_-
		\end{array}
		\right.
\end{align}
which restricts to a linear isomorphism
\begin{align}
\mathcal{Q}_- : \Dist(\Sn ; \Sym^m_0 \T^*\Sn ) \to \Dist( S\Hn ; \Sym^m_0 \E^* ) \cap \ker A \cap \ker \nm.
\end{align}
Moreover, suppose we define
\begin{align}
u:= (\Phi_-)^\lambda \mathcal{Q}_- \omega,
\qquad
\lambda\in\C,
\omega\in\Dist(\Sn ; \Sym^m_0 \T^*\Sn ),
\end{align}
then $u$ enjoys, due to
\eqref{eq:GEquivarianceOfTandU}
and
\eqref{eq:GEquivarianceOfTauMinus},
the following equivariance property for $\gamma\in G$
\begin{align}
\left(
	\gamma^* 
	(\Phi_-)^\lambda 
	\mathcal{Q}_- 
	\omega
\right)_{(x,\xi)} 
( \eta_1,\dots,\eta_m)
=
(\Phi_-)^\lambda_{(x,\xi)} 
\left(
	(T_\gamma)^{\lambda+m} 
	U_\gamma^*
	\omega
\right)_{B_-(x,\xi)} 
( {\tau_-} \eta_1, \dots , {\tau_-} \eta_m)
\end{align}
where $\eta_i \in \E_{(x,\xi)}$.
So 
$\gamma^* u=u$
if and only if, for 
$y\in\Sn$,
\begin{align}\label{eq:GEquivarianceOfOmega}
U_\gamma^* \omega(y) = T_\gamma (y)^{-\lambda-m} \omega(y).
\end{align}

\subsection{Upper half-space model}

Hyperbolic space is diffeomorphic to the upper half-space model
$\Un:= \R^+\times \R^n$.
We take its closure
$\overline \Un$
by considering
$\Un\subset \R^{n+1}$.
Using coordinates
$x=(\rho,y)$
for
$\rho\in\R^+$,
$y\in\R^n$
the metric takes the form
\begin{align}
g = \frac{d\rho^2 + h}{\rho^2}
\end{align}
where
$h$
is the standard metric on
$\R^n$.

In this model of hyperbolic space, the map $\tau_-^{-1}$ has been explicitly calculated in
\cite[Appendix A]{g-moroianu-park}
under the guise of parallel transport in the 0-calculus of Melrose.
For
$y'\in \R^n$,
$x=(\rho,y)\in \Un$,
we write
$\xi_-:=\xi_-(x,y')$
and
$r:=y-y'$.
Then
\begin{align}
\tau_-^{-1} : \left\{	 \begin{array}{rcl}
			\T_{y'} \R^n & \to & \E_{(x,\xi_-)} \\
			\partial_{y_i} & \mapsto & \rho\left( \frac{-2\rho^2 r_j}{\rho^2+r^2} \drr + \sum_{j=1}^n \left(\delta_{ij} - \frac{2r_ir_j}{\rho^2+r^2} \right) \partial_{y_j} \right)
				\end{array}
		\right.
\end{align}
Therefore 
$\tau_-^* dy_i=\rho^{-1}dy_i$
if 
$r=0$ 
and in general, for fixed 
$y'$ 
and variable 
$x$,
\begin{align}\label{eqn:evenexpansiontau}
\tau_-^* dy_i = \rho^{-1} \left( b \, \rho d\rho + \sum_{j=1}^n b_{ij} dy_j \right)
\end{align}
for
$b,b_{ij} \in C^\infty_\even(\overline\Un)$
(that is, 
$b,b_{ij}$ are functions of 
$\rho^2$
rather than simply
$\rho$).

The Poisson kernel reads (continuing to use the notation from the previous paragraph)
\begin{align}
P(x,y') = \frac{\rho}{\rho^2+r^2}(1+|y|^2)
\end{align}
and so 
$\rho^{-1}P(x,y')$
is even in 
$\rho$
and, for fixed
$y'$,
is smooth on 
$\overline\Un$ 
away from
$x= (0,y')$.

\subsection{Convex cocompact quotients}

Consider a discrete subgroup 
$\Gamma$
of
$G=\SO_0(1,n+1)$
which does not contain elliptic elements.
Denote by 
$K_\Gamma$
the limit set of
$\Gamma$.
Via the compactification
$\overline\Hn = \Hn \sqcup \Sn$,
the limit set is the
the set of accumulation points of an arbitrary 
$\Gamma$-orbit,
and is a closed subset of
$\Sn$.
The hyperbolic convex hull of all geodesics in
$\Hn$
whose two endpoints both belong to $K_\Gamma$
is termed the convex hull.
The quotient of the convex hull by 
$\Gamma$
gives the convex core of 
$\Gamma\backslash\Hn$,
that is,
the smallest convex subset of 
$\Gamma\backslash\Hn$
containing all closed geodesics of
$\Gamma\backslash\Hn$.
The group
$\Gamma$ 
is called convex cocompact if
its associated convex core is compact.

Let 
$\Gamma$ 
be convex cocompact and define 
$\Xone:=\Gamma \backslash \Hn$
denoting the canonical projection by
$\pi_\Gamma : \Hn \to \Xone$.
Then
$S\Xone = \Gamma \backslash S\Hn$
(with canonical projection also denoted by
$\pi_\Gamma$).
The constructions of the previous subsections descend to constructions on
$\Xone$
and
$S\Xone$.

Furthermore, denote by
$\Omega_\Gamma\subset \Sn$
the discontinuity set of
$\Gamma$.
Then
$\Omega_\Gamma=\Sn \backslash K_\Gamma$
and
$\overline \Xone = \Gamma \backslash (\Hn \sqcup \Omega_\Gamma)$.
Denote by
$\delta_\Gamma$
the Hausdorff dimension of the limit set
$K_\Gamma$.

We introduce the outgoing tail 
$K_+\subset S\Xone$ as
$
K_+ 
:=
\pi_\Gamma \left( 
				B_-^{-1} \left(
							K_\Gamma
						\right)
			\right)
$ 
and remark that this may be interpreted as the set of points 
$(x,\xi)\in S\Xone$
such that
$\pi_S(\varphi_t(x,\xi))$
does not tend to
$\partial\overline\Xone$
as
$t\to-\infty$.
Using the outgoing tail, we define the following restriction of the unstable dual bundle
$
E_+^* :=  \left. E^{*u} \right|_{K_+}.
$ 


\section{Ruelle Resonances}\label{sec:ruelle}

The operator
$A$
acts on 
$\Sym^m\E^*$ 
above 
$S\Xone$. 
For 
$\lambda\in \C$
with
$\Re\lambda>0$,
the operator
$(A+\lambda)$
has an inverse acting on
$L^2(S\Xone;\Sym^m\E^*)$.
By
\cite{dyatlov-g},
this inverse admits a meromorphic extension to
$\C$
as a family of bounded operators
\begin{align}
\Resolvent{A,m}{\lambda} : C_c^\infty(S\Xone ; \Sym^m\E^*) \to \Dist(S\Xone ; \Sym^m\E^*).
\end{align}
Near a pole 
$\lambda_0$, 
called a Ruelle resonance 
(of tensor order $m$),
the resolvent may be expressed as
\begin{align}
\Resolvent{A,m}{\lambda} 
	= 	\ResolventHol{A,m}{\lambda} 
		+ \sum_{j=1}^{J(\lambda_0)} 	\frac{(-1)^{j-1}(A+ \lambda_0)^{j-1} \Projector{A,m}{\lambda_0}}{(\lambda-\lambda_0)^{-j}}
\end{align}
where the image of the finite rank projector 
$\Projector{A,m}{\lambda_0}$ 
is called the space of generalised Ruelle resonant states 
(of tensor order $m$). 
It is denoted
\begin{align}
\Resonant{A,m}{}{\lambda_0}
&:=
\Im \left( \Projector{A,m}{\lambda_0} \right) \\
&\phantom{:}=
\setright{
	u \in \Dist (S\Xone ; \Sym^m \E^*) 
	}{
	\supp(u)\subset K^+ 
	,\, 
	\WF(u) \subset E_+^* 
	,\, 
	(A+\lambda_0)^{J(\lambda_0)}u=0
	}
\end{align}
(This characterisation in terms of support and wavefront properties being given in 
\cite{dyatlov-g}.)
We filter this space by declaring 
\begin{align}
\Resonant{A,m}{j}{\lambda_0}
:=
\setleft{
u \in \Resonant{A,m}{}{\lambda_0} 
}{
(A+ \lambda_0)^j u =0
}
\end{align}
saying that such states are of Jordan order (at most) 
$j$. 
Then
\begin{align}
\Resonant{A,m}{}{\lambda_0} = \cup_{j \ge1} \Resonant{A,m}{j}{\lambda_0}
\end{align}
and the space of Ruelle resonant states is 
$\Resonant{A,m}{1}{\lambda_0}$.

\subsection{Band structure}
Consider now 
$A$ 
acting on 
$\Sym^0\E^*$.
Let 
$\lambda_0$ 
be a Ruelle resonance 
(of tensor order $0$)
and consider (a non-zero) 
$ u\in \Resonant{A,0}{}{\lambda_0}$.
As Ruelle resonances
(of arbitrary tensor order)
are contained in
$\setleft{ \lambda \in \C }{ \Re\lambda\le0 }$,
the commutator relation
$[A,\dm]=-\dm$
implies that
there exists
$m\in\N_0$
such that
$(\dm)^m u \neq 0$
and
$(\dm)^{m+1} u =0$.
We say that 
$u$ 
is in the 
$m$\textsuperscript{th}
band. Precisely, we define
\begin{align}
\V{m}{j}{\lambda_0}
:=
\setleft{
	u \in \Resonant{A,0}{j}{\lambda_0}
	}{
	u \in \ker (\dm)^{m+1}
	}
\end{align}
The
$m$\textsuperscript{th}
band may then be considered the quotient
$\V{m}{j}{\lambda_0} / \V{m-1}{j}{\lambda_0}$
whence
\begin{align}\label{eq:FilterRuelleIntoBands}
\Resonant{A,0}{j}{\lambda_0} 
&= 
\bigoplus_{m\in \N_0}				\left( \V{m}{j}{\lambda_0} / \V{m-1}{j}{\lambda_0} \right).
\end{align}
Propositions~\ref{prop:IdentifyBandsWithSymmetricPowerOfE} and~\ref{prop:lambda=m-case} identify these bands with Ruelle resonances of tensor order $m$. This identification requires an inversion of horosphere operators presented in 
\cite[Section 4.3]{dfg}.
Specifically, the following lemma is a restatement of the final calculations performed in said section using the notation of the current article.
\begin{lemma}\label{lem:HorocycleInversion}
Consider 
$u\in \Dist(S\Hn; \Sym^m \E)\cap \ker (\nm)$ 
decomposed such that 
$u=\sum_{k} \adjtrace^k u^{(m-2k)}$ 
for 
$u^{(m-2k)} \in \Dist(S\Hn; \Sym^{m-2k}_0 \E)\cap \ker (\nm)$.
Set
$r:=m-2k$.
Then on $u^{(m-2k)}$,
\begin{align}
(\dm)^{m} (\Deltap)^k (\divp)^r
=
\adjtrace^k P_{r,k}(A)
\end{align}
where $P_{r,k}(A)$ is the following polynomial
\begin{align}
P_{r,k}(A)
=
2^{k+r} m! (r!)^2
\prod_{j=1}^k	
(A+r+j-1)(-2A+(n-2j))
\prod_{j=1}^r
(A-n-j+2).
\end{align}
\end{lemma}
\qed

One deduces that if we take 
$\lambda_0\in\C\backslash ( -\frac n2 - \frac 12 \N_0 )$
with 
$\Re\lambda\le -1$,
and if we take 
$m\in \N$,
$r,k\in \N_0$
with 
$m=r+2k$,
then the value of the polynomial 
$P_{r,k}(-(\lambda_0+m))$
is non zero,
except in the single situation
$m\in 2\N$,
$r=0, k=m, \lambda_0+m=0$.

\begin{proposition}\label{prop:IdentifyBandsWithSymmetricPowerOfE}
Consider 
$\lambda_0\in\C\backslash ( -\frac n2 - \frac 12 \N_0 )$,
a Ruelle resonance
with 
$\Re\lambda_0\le-1$. 
Consider also 
$m\in\N$ 
such that 
$\Re\lambda_0 + m\le 0$. 
%
%
%
Further, exclude the case 
$m$ 
even with 
$\lambda_0+m=0$.
Under these assumptions, we obtain the following short exact sequence
\begin{align}
\begin{array}{ l l l l 	l l l l 		l }
0
&\longrightarrow
&\V{m-1}{j}{\lambda_0}
&\longrightarrow
&\V{m}{j}{\lambda_0}
&\xrightarrow{(\dm)^m}
&\Resonant{A,m}{j}{\lambda_0+m} \cap \ker \nm
&\longrightarrow
&0
\end{array}
\end{align}
\end{proposition}
\begin{proof}
Denote by 
$W_m^j(\lambda_0+m)$
the third space in the sequence
$\Resonant{A,m}{j}{\lambda_0+m} \cap \ker \nm$.
The non-trivial step is showing surjectivity of 
$(\dm)^m$. 
We decompose 
$W^j_m(\lambda_0+m)$ 
into eigenspaces of 
$\adjtrace\trace$. 
In particular we denote 
\begin{align}
W^j_{m,k}(\lambda_0+m) := \adjtrace^k \left( W^j_{m-2k}(\lambda_0+m) \cap \ker \trace \right). 
\end{align}
By
Lemma~\ref{lem:HorocycleInversion}, 
there exists differential operators (linear of order $m$)
\begin{align}
K_k : W^j_{m,k}(\lambda+m) \to \V{m}{j}{\lambda_0}
\end{align}
such that 
$(\dm)^m\circ K_k 
= P_{m-2k,k}(A)$
where 
$P_{m-2k,k}=P_{r,k}$ 
is the polynomial from
Lemma~\ref{lem:HorocycleInversion}.

As
$W^j_{m,k}(\lambda_0+m)$
is finite dimensional, it suffices to show injectivity of 
$(\dm)^m\circ K_k$
which we do by induction on 
$j$.
Consider
$j=1$
in which case
$(\dm)^m\circ K_k=P_{m-2k,k}(-(\lambda_0+m))$ 
on
$W^1_{m,k}(\lambda_0+m)$ 
which is non-zero by the comment following the preceding lemma.

Consider now 
\begin{align}
u\in W^j_{m,k}(\lambda_0+m) \cap \ker ((\dm)^m\circ K_k).
\end{align}
By considering again a decomposition of the form
$u=\sum_{k=0}^{\lfloor \frac m2 \rfloor} \adjtrace^k u^{(m-2k)}$,
then the fact that
$(\dm)^m\circ K_k$
is a polynomial in 
$A$, 
implies that it commutes with
$(A+\lambda_0+m)$
hence
\begin{align}
(A+\lambda_0+m)u^{(m-2k)}\in W^{j-1}_{m-2k}(\lambda_0+m) \cap \ker \trace \cap \ker ((\dm)^m\circ K_k)
\end{align}
which by the inductive hypothesis forces 
$u\in\ker(A+\lambda_0+m)$
and the case 
$j=1$ 
now implies 
$u=0$.
\end{proof}

\begin{proposition}\label{prop:lambda=m-case}
Consider 
$\lambda_0\in-2\N \backslash (-\tfrac n2 - \tfrac 12\N_0)$,
a Ruelle resonance and set
$m:=-\lambda_0$.
Then
\begin{align}
\Resonant{A,m}{j}{0} \cap \ker \nm = 0
\end{align}
so in this case also, there is trivially a short exact sequence as in 
Proposition~\ref{prop:IdentifyBandsWithSymmetricPowerOfE}.
\end{proposition}
\begin{proof}
It suffices to prove the statement for 
$j=1$.
Suppose
$u\in \Resonant{A,m}{1}{0} \cap \ker \nm$
non-zero and decompose 
$u =: \adjtrace^k u^{(m-2k)}$ 
for
$u^{(m-2k)} \in \Resonant{A,m-2k}{1}{0} \cap \ker \trace \cap \ker  \nm$.
Consider first
$ u^{(0)} $. 
This is a Ruelle resonant state (of tensor order
$0$)
on
$S\Xone$
but by
\cite{dyatlov-g}
the real part of a
Ruelle resonance of tensor order 
$0$
is not greater than
$\delta_\Gamma-n<0$.
Considering the other components of 
$u$,
define
$\varphi^{(m-2k)} : = \pi_{0*} u^{(m-2k)}$ 
for $m-2k\neq 0$
($\pi_{0*}$
being defined in Section~\ref{sec:poisson}).
By Proposition~\ref{prop:RuelleToQuantum}
this is an isomorphism
\begin{align}
\pi_{0*}:
\Resonant{A,m-2k}{1}{0} \cap \ker \trace \cap \ker \nm
\to
\Resonant{\Delta,m-2k}{1}{n}.
\end{align}
From
\cite[Lemma 8.2]{dairbekov-sharafutdinov}
and the discussion preceding
\cite[Lemma 6.1]{dfg}
the
$\Lsections$
spectrum of
$\n^*\n$
acting
on
$\Sym^{m-2k}_0\T^*\Xone$
(for $m-2k\neq0$)
is bounded below
by
$(n+m-2k-1)$.
However
$\varphi^{(m-2k)} \in \ker (\n^*\n -(m-2k))$
and by
Lemma~\ref{lem:AssOfQRes},
$\varphi^{(m-2k)} \in \Lsections(\Xone;\Sym^{m-2k}_0\T^*\Xone)$.
This
forces
$\varphi^{(m-2k)}=0$
as
$m-2k< n+m-2k-1$.
\end{proof}

To finish this section, we consider the decomposition of the set of vector-valued generalised resonant states considered in this subsection into eigenspaces of $\adjtrace\trace$. Then
\begin{align}\label{eq:DecomposeMTensorsTrace}
\Resonant{A,m}{j}{\lambda_0+m} \cap \ker \nm
=
	\bigoplus_{k=0}^{\lfloor \frac m2 \rfloor} \adjtrace^k 
		\left( \Resonant{A,m-2k}{j}{\lambda_0+m} 
			\cap \ker \trace \cap \ker \nm
		\right)
\end{align}
as $A$ commutes with the Lefschetz-type operators, and the condition 
$\ker\nm$ 
is conserved (which may be concluded from considering the form of 
$\nm$ 
acting on
$\Dist(G;\Sym^m\R^n)/\mathfrak{m}$).

\section{Quantum Resonances}\label{sec:quantum}

This section includes the principal calculation of this paper, performed in
Lemma~\ref{lem:AssOfQRes}. It characterises symmetric tensor valued generalised quantum resonant states via their asymptotic structure. Quantum resonant states are defined using the meromorphic extension of the resolvent of the Laplacian obtained in 
\cite{h:quantum}
which is based on Vasy's method
\cite{v:ml:inventiones,v:ml:forms}.
In order to prove
Lemma~\ref{lem:AssOfQRes},
a mere knowledge of the meromorphic extension (of the resolvent of the Laplacian) does not seem to suffice. Indeed the proof presented requires meromorphic extensions of resolvents of various operators constructed in
\cite{h:quantum}. These are recalled in the following subsection.

\subsection{Vasy's operator on even asymptotically hyperbolic manifolds}

Consider the Lorentzian cone
$\Mone := \R^+_\s \times \Xone$
with Lorentzian metric
$\eta = -d\s\otimes d\s + \s^2 g$
where
$(\Xone,g)$ is even asymptotically hyperbolic
\cite[Definition 5.2]{g:duke}
with Levi-Civita connection 
$\n$.
(Convex cocompact quotients of hyperbolic space being the model geometry for such manifolds.)
Symmetric tensors decompose 
\begin{align}
\Sym^m \T^* \Mone = \bigoplus_{k=0}^m a_k \, (\dss)^{m-k} \sp \Sym^k  \T^*\Xone,
\qquad
a_k := \tfrac{1}{\sqrt{(m-k)!}}
\end{align}
and the (Lichnerowicz) d'Alembertian
$\DeltaAmb$ 
acts on symmetric 
$m$-tensors. 
A particular conjugation by 
$\s$ 
of 
$\s^2 \DeltaAmb$ 
behaves nicely relative to the preceding decomposition giving the operator
\begin{align}
\QVasyUp := \s^{\frac n2 - m+2} \DeltaAmb \s^{-\frac n2 +m} = \n^*\n + (\sds)^2 + \DVasyUp + \ZeroOrderVasyUp
\end{align}
for a first order differential operator 
$ \DVasyUp + \ZeroOrderVasyUp$ 
on 
$\Sym^m \T^*\Mone$. 
(Above 
$\sds$ 
is considered a Lie derivative 
and, along with
$\n^*\n$,
acts diagonally on each factor
$ \left(\dss\right)^{m-k} \sp \Sym^k  \T^*\Xone$.)
The b-calculus of Melrose 
\cite{melrose:aps}
permits this operator to be pushed to a family of operators, denoted
$\QVasyDOWNL$,
(holomorphic in the complex variable 
$\lambda$)
acting on
$\oplus_{k=0}^m\Sym^k \T^*\Xone$
above
$\Xone$
which takes the form
\begin{align}
\QVasyDOWNL = \n^*\n + \lambda^2 + \DVasyDown + \ZeroOrderVasyDown
\end{align}
for a first order differential operator
$\DVasyDown + \ZeroOrderVasyDown$.
(A more precise description of
$\DVasyDown + \ZeroOrderVasyDown$
will be given shortly.)

Consider a boundary defining function,
$\rho$,
for the conformal compactification 
$\overline \Xone$.
Near
$Y:=\partial\overline\Xone$,
say on
$U:=(0,1)_\rho\times Y$,
the metric may be written
\begin{align}
g
=
\frac{d\rho^2+ h}{\rho^2}
\end{align}
where 
$h$ 
is a family of Riemannian metrics on 
$Y$ 
smoothly parametrised by
$\rho\in[0,1)$
whose Taylor expansion at 
$\rho=0$ 
contains only even powers of 
$\rho$. 
Again consider the Lorentzian cone
$\Mone=\R^+_\s \times \Xone$ 
with metric
$\gamb$.
The metric $\gamb$ degenerates at 
$\rho=0$
however under the change of coordinates
\begin{align}
\t:=\s/\rho,
\qquad
\mu:=\rho^2
\end{align}
the metric takes the following form
on
$\RplusT \times (0,1)_\mu \times Y$
\begin{align}
\gamb = -\mu d\t \otimes d\t - \tfrac12 \t ( d\mu \otimes d\t + d\t\otimes d\mu) + \t^2 h.
\end{align}
We extend the manifold
$\Xone$
to a slightly larger manifold
$\Xthree:=((-1,0]_\mu\times Y) \sqcup \Xone$
and use
$\mu$
to provide a smooth structure explained precisely in 
\cite[Section 2]{h:quantum}.
(Importantly, the chart
$(-1,1)_\mu\times Y$
provides smooth coordinates 
near 
$\partial \overline\Xone$
in
$\Xthree$.)
The ambient Lorentzian metric
$\gamb$
is also extended to
$\Mthree := \RplusT\times \Xthree$
by extending 
$h$ 
to a family of Riemannian metrics on
$Y$
smoothly parametrised by
$\mu \in (-1,1)$.

We require a notion of even sections on
$\overline\Xone$.
We declare 
$C^\infty_\even(\overline\Xone)$
to be the restriction of 
$C^\infty(\Xthree)$
to
$\overline\Xone$.
Similarly, for a vector bundle which is defined over 
$\Xthree$,
notably
$\Sym^m\T^*\Xthree$,
the notion of even sections
is defined as the restriction to 
$\overline\Xone$
of smooth sections over
$\Xthree$.

We now follow the recipe given in the first paragraph of this subsection.
The Lichnerowicz d'Alembertian
$\DeltaAmb$ 
acts on symmetric $m$-tensors above $\Mthree$. Conjugating $\t^2\DeltaAmb$ provides 
\begin{align}
\PVasyUp := \t^{\frac n2 - m+2} \DeltaAmb \t^{-\frac n2 +m}
\end{align}
The b-calculus pushes this operator to a family of operators (holomorphic in the complex variable $\lambda$), termed ``Vasy's operator" and denoted 
\begin{align}
\PVasyDOWNL \in \mathrm{Diff}^2 (\Xthree ; \oplus_{k=0}^m \Sym^k \T^*\Xthree).
\end{align}
It is elliptic on 
$\Xone$
and hyperbolic on
$\Xthree\backslash \overline\Xone$.
On
$U$,
the two families are related 
\begin{align}
\PVasyDOWNL = \rho^{- \lambda - \frac n2 + m - 2} J \QVasyDOWNL J^{-1} \rlnm
\end{align}
for
$J\in C^\infty( \Xone ; \End( \oplus_{k=0}^m \Sym^k \T^* \Xone ))$ 
whose entries are homogeneous polynomials in 
$\drr\sp$,
upper triangular in the sense that
$J(\Sym^{k_0} \T^*\Xone) \subset \oplus_{k=k_0}^m \Sym^k \T^*\Xone$,
and whose diagonal entries are the identity.

There are meromorphic inverses with finite rank poles for the operators
$\PVasyDOWNL$
and
$\QVasyDOWNL$.
(Using  
$\eta$ to provide a notion of regularity for
sections
of
$\oplus_{k=0}^m \Sym^k \T^*\Xthree$
and
microlocal analysis, including propogation of singularities and radial point estimates, in order to solve a Fredholm problem.)
We denote respectively these meromorphic inverses by
\begin{align}
\Resolvent{\PVasyDown,m}{\lambda}
:
C_c^\infty (\Xthree ; \oplus_{k=0}^m \Sym^k \T^*\Xthree)
\to
C^\infty( \Xthree ; \oplus_{k=0}^m \Sym^k \T^*\Xthree)
\end{align}
and
\begin{align}
\Resolvent{\QVasyDown,m}{\lambda}
:
C_c^\infty (\Xone ; \oplus_{k=0}^m \Sym^k \T^*\Xone)
\to
\rlnm  \oplus_{k=0}^m \rho^{-2k}  C^\infty_\even (\overline\Xone ; \Sym^k \T^*\Xone).
\end{align}

To finish this subsection we restrict to the case where
$\Xone$ is a convex cocompact quotient of hyperbolic space. 
Lemma~\ref{lem:AssOfQRes}
does not require a complete description of 
$\QVasyDOWNL$
however its form upon restriction to
$\Sym^m_0 \T^*\Xone$
is required. Precisely, we have
\begin{align}
\left. \QVasyDOWNL \right|_{\Sym^m_0 \T^*\Xone}
=
\twovector{\n^*\n + \lambda^2 - \frac{n^2}{4} - m}{ -2 \div }
:
C^\infty(\Xone ; \Sym^m_0 \T^*\Xone )
\to
C^\infty(\Xone ; \oplus_{k=m-1}^m \Sym^k_0 \T^*\Xone)
\end{align} 
which upon setting $s := \lambda + \frac n2$ provides
\begin{align}
\left. \QVasyDown_{s-\frac n2} \right|_{\Sym^m_0 \T^*\Xone}
=
\twovector{\n^*\n - s(n-s) - m}{ -2 \div }.
\end{align}
In a similar spirit we record that
\begin{align}
\left. J \right|_{\oplus_{k=m-1}^m \Sym^k_0 \T^*\Xone} = \twomatrix{1}{\drr\sp}{0}{1}.
\end{align}

\subsection{Quantum resonances for convex cocompact quotients}

The rough Laplacian 
$\n^*\n$
acts on 
$\Sym^m_0\T^*\Xone$.
For 
$s\in\C$
with
$s\gg 1$,
the operator
$\n^*\n-s(n-s)-m$
has an inverse acting on
$\Lsections(\Xone; \Sym^m_0 \T^*\Xone)$.
Since
$\Xone$ is locally hyperbolic space,
$\n^*\n$
commutes with the divergence operator
$\div$.
This property is key to proving the meromorphic extension of the inverse
\cite[Theorem 1.4]{h:quantum}.
Precisely,
the inverse of 
$\n^*\n-s(n-s)-m$,
written
$\Resolvent{\Delta,m}{s}$,
admits,
upon restriction to 
$\Sym^m_0\T^*\Xone \cap \ker\delta$,
a meromorphic extension from
$\Re s\gg 1$
to
$\C$
as a family of bounded operators
\begin{align}
\Resolvent{\Delta,m}{s} 
: C_c^\infty(\Xone;\Sym^m_0 \T^*\Xone) \cap \ker \div 
\to
\rho^{s -m} C^\infty_\even(\overline \Xone;\Sym^m_0 \T^*\Xone) \cap \ker \div. 
\end{align}
(Here
$\rho$ is an even boundary defining function providing the conformal compactification
$\overline \Xone$.)
Near a pole 
$s_0$, 
called a quantum resonance,
the resolvent may be written
\begin{align}
\Resolvent{\Delta,m}{s} 
= 
\ResolventHol{\Delta,m}{s} 
+ 
\sum_{j=1}^{J(\lambda_0)} \frac{ 	(\n^*\n-s_0(n-s_0)-m)^{j-1} \Projector{\Delta,m}{s_0}}
						{	(s(n-s)-s_0(n-s_0))^j}
\end{align}
where the image of the finite rank projector 
$\Projector{\Delta,m}{\lambda_0}$ 
is called the space of generalised quantum resonant states (of tensor order $m$)
\begin{align}
\Resonant{\Delta,m}{}{s_0}
:=
\Im \left( \Projector{\Delta,m}{s_0} \right).
\end{align} 
We filter this space by declaring
\begin{align}
\Resonant{\Delta,m}{j}{ s_0}
:=
\setleft{
\varphi \in \Resonant{\Delta,m}{}{ s_0} 
}{
(\n^*\n- s_0(n- s_0)-m)^j \varphi =0
}
\end{align}
saying that such states are of Jordan order (at most) $j$. Then
\begin{align}
\Resonant{\Delta,m}{}{ s_0} = \cup_{j \ge1} \Resonant{\Delta,m}{j}{ s_0}
\end{align}
and the space of quantum resonant states is $\Resonant{\Delta,m}{1}{ s_0}$.

\begin{lemma}\label{lem:AssOfQRes}
For 
$s_0\in \C$
with
$s_0 \neq \frac n2$, 
generalised quantum resonant states 
$\Resonant{\Delta,m}{j}{s_0}$
are precisely identified with 
\begin{align}
\setleft{
	\varphi \in \bigoplus_{k=0}^{j-1} \, \rho^{s_0-m} (\log \rho)^k \,C^\infty_\even(\overline \Xone;\Sym^m_0\T^* \Xone)
	}{
	\varphi \in \ker (\n^*\n - s_0(n-s_0)-m)^j \cap \ker \div 
}.
\end{align}
\end{lemma}

\begin{proof}
We introduce the short-hand
\begin{align}
\AVasyDown_{s} : = (\n^*\n - s(n-s) - m).
\end{align}

That a generalised resonant state has the prescribed form is reasonably direct. Indeed given 
$\varphi \in \Im \left(  \Projector{\Delta,m}{s_0} \right)$
there exists 
$\psi \in C_c^\infty( \Xone;\Sym^m_0 \T^* \Xone)$ 
(which is divergence-free) such that 
$\varphi = \Residue{s_0}{\Resolvent{\Delta,m}{s} \psi}$. 
By
\cite[Theorem 1.4]{h:quantum},
we may write
\begin{align}
\Resolvent{\Delta,m}{s} \psi =: \rho^{s-m} \Psi_s \in \ker\div
\end{align}
for 
$\Psi$
a meromorphic family taking values in 
$C^\infty_\even(\overline \Xone;\Sym^m_0\T^* \Xone)$. 
Supposing the specific Jordan order of 
$\varphi$ 
to be 
$j\le J(s_0)$, 
equivalently 
$\AVasyDown^{j-1}_{s_0} \varphi\neq 0$ 
and 
$\varphi \in \ker \AVasyDown^j_{s_0}$, 
implies 
$\Psi$ 
has a pole of order 
$j$ 
at 
$s_0$. 
Expanding 
$\rho^{s-m}$ 
and 
$\Psi_s$ 
in Taylor and Laurent series about 
$s_0$ 
respectively gives
\begin{align}
\Resolvent{\Delta,m}{s} \psi
=
\left(
	\rho^{s_0-m} \sum_{k=0}^{j-1} (\log \rho)^k \frac{(s-s_0)^k}{k!} + O((s-s_0)^j)
\right)
\left(
	\Psi^\mathrm{Hol}_s + \sum_{k=0}^{j} \frac{\Psi^{(k)}}{(s-s_0)^k}
\right)
\end{align}
with 
$\Psi^\mathrm{Hol}$ 
(a holomorphic family) and 
$\Psi^{(k)}$ 
taking values in 
$C^\infty_\even(\overline \Xone;\Sym^m_0\T^* \Xone)$.
Extracting the residue gives the result that
\begin{align}
	\varphi \in \left(
				 \oplus_{k=0}^{j-1} \,\rho^{s_0-m} (\log \rho)^k \,C^\infty_\even(\overline \Xone;\Sym^m_0\T^* \Xone)
			\right)
			 \cap \ker \div.
\end{align}

For the converse statement we initially follow \cite[Proposition 4.1]{ghw}. Suppose 
$\varphi \in \ker \AVasyDown_{s_0}^{j}$
trace-free, divergence-free, and takes the required asymptotic form. We may suppose 
$\AVasyDown_{s_0}^{j-1}\varphi \neq 0$.
Set 
\begin{align}
\varphi^{(1)}
:= 
\AVasyDown_{s_0}^{j-1} \varphi 
\in \rho^{s_0-m} C^\infty_\even(\overline \Xone;\Sym^m_0\T^* \Xone) \cap\ker\div  .
\end{align}
For 
$k\in\{2,\dots, j\}$, 
there exist polynomials 
$p_{k,l}$ 
such that upon defining
\begin{align}
\varphi^{(k)} 
:= 
(n-2s_0)^{k-1} \AVasyDown_{s_0}^{(j-k)}\varphi + \sum_{\ell=1}^{k-1} p_{k,l}(n-2s_0) \AVasyDown_{s_0}^{(j-k+\ell )}\varphi
\in \ker\trace\cap\ker\div
\end{align}
we satisfy the condition, for 
$k\in\{1,\dots,j\}$,
\begin{align}\label{eqn:condition-phi(k)}
\AVasyDown_{s_0} \varphi^{(k)} - (n-2s_0)\varphi^{(k-1)} + \varphi^{(k-2)} =0
\end{align}
(with the understanding that $\varphi^{(0)}=\varphi^{(-1)}=0$).
Note that such a condition appears upon demanding 
\begin{align}
\AVasyDown_s
\varphi_s = O((s-s_0)^j),
\qquad
\varphi_s := \sum_{k=1}^j \varphi^{(k)} (s-s_0)^{k-1}
\end{align}
Define
\begin{align}
\Phi_s := \sum_{k=1}^j \Phi^{(k)}(s-s_0)^{k-1},
\qquad
\Phi^{(k)} := \rho^{-s_0+m} \sum_{\ell=0}^{k-1} \tfrac{(-\log \rho)^\ell}{\ell !} \varphi^{(k-\ell)}.
\end{align}
We claim that 
\begin{align}\label{eqn:ClaimAsymptoticPhi}
\Phi^{(k)}\in C^\infty_\even(\overline \Xone ; \Sym^m_0 \T^* \Xone).
\end{align} 
As 
$\Phi^{(k)}$ 
a priori belongs in the space 
$\oplus_{\ell=0}^{k-1} \,(\log \rho)^\ell \,C^\infty_\even(\overline \Xone;\Sym^m_0\T^* \Xone)$, 
it suffices to observe that
\begin{align}
\PVasyDown_{s_0-\frac n2}\Phi^{(k)} \in C^\infty_\even(\overline \Xone ; \oplus_{k={m-1}}^m \, \Sym^k \T^* \Xone)
\end{align}
where
\begin{align}
\rho^2 \PVasyDown_{s_0-\frac n2} \Phi^{(k)} =  \twomatrix{1}{\drr\sp }{0}{1} \rho^{-s_0+m}
\twovector{\AVasyDown_{s_0}}{-2\div} \rho^{s_0-m} \Phi^{(k)}.
\end{align}
We perform the required calculation in the collar neighbourhood 
$U=(0,1)_\rho\times Y$
where the metric is of the form
$g = \rho^{-2}(d\rho^2 + h)$
and with a frame 
$\ensemble{dy^i}{1}{i}{n}$ 
for 
$\T^*Y$.
Define 
$\rho^{-2} B\in C^\infty_\even( \Xone;\End(\T^*Y))$ 
by 
$Bdy^i := \sum_{jk} \frac12 (h^{-1})^{ij}(\rdr h_{jk}) dy^k$ 
and extend it to 
$\rho^{-2} B\in C^\infty_\even( \Xone;\End(\T^* \Xone))$ 
as a derivation with 
$B d\rho:=0$. 
The Laplacian, on functions, takes the form
\begin{align}\label{eqn:lapfunctions}
\Delta = -(\rdr)^2 + \rho^2 \Delta_h + (n-\tr_h B) \rdr.
\end{align}
We calculate
$\rho^2 \PVasyDown_{s_0-\frac n2} \Phi^{(k)}$.
The first tedious step is
\begin{align}
&\rho^{-s_0+m} \AVasyDown_{s_0} \rho^{s_0-m} \Phi^{(k)} \\
&= \rho^{-s_0 + m} \left( \Delta - s_0(n-s_0) - m \right) \sum_{\ell=0}^{k-1} \frac{(-\log \rho)^\ell}{\ell !} \varphi^{(k-\ell)} \\
&=\rho^{-s_0 + m} \sum_{\ell=0}^{k-1} \frac{1}{\ell !} 
	\left( 
	(-\log \rho)^\ell \AVasyDown_{s_0} \varphi^{(k-\ell)} 
	- 2 \tr_g \left(
			\n (-\log \rho)^\ell \otimes \n \varphi^{(k-\ell)}
			\right)
	+ (\Delta (-\log \rho)^\ell ) \varphi^{(k-\ell)} 
	\right)\\
&=\rho^{-s_0 + m} \sum_{\ell=0}^{k-1} \frac{1}{\ell !} 
	\left( 
	(-\log \rho)^\ell \AVasyDown_{s_0} \varphi^{(k-\ell)} 
	- 2\tr_g \left(
			\n (-\log \rho)^\ell \otimes \n \varphi^{(k-\ell)}
			\right)
	+ \left(\Delta (-\log \rho)^\ell \right) \varphi^{(k-\ell)} 
	\right)
\end{align}
and we split this calculation up further into three parts. 
Treating the first part with \eqref{eqn:condition-phi(k)},
\begin{align}
&\rho^{-s_0+m}\sum_{\ell = 0}^{k-1} 		
		\frac{(-\log \rho)^\ell}{\ell !} \AVasyDown_{s_0}  \varphi^{(k-\ell)} \\
&= \rho^{-s_0+m}\sum_{\ell = 0}^{k-1} 	
	\frac{(-\log \rho)^\ell}{\ell !} \left( 
						(n-2s_0) \varphi^{(k-1-\ell)} - \varphi^{(k-2-\ell)} 
						\right) \\
&=	(n-2s_0) \Phi^{(k-1)} - \Phi^{(k-2)}.
\end{align}
Treating the second part directly
\begin{align}
&\rho^{-s_0+m}\sum_{\ell = 0}^{k-1} 		\frac{1}{\ell !} 
		\left(
		- 2 \tr_g \left(
			\n (-\log \rho)^\ell \otimes \n \varphi^{(k-\ell)}
			\right)
		\right)\\
&=\rho^{-s_0+m}\sum_{\ell = 0}^{k-1} 		\frac{(-\log \rho)^{\ell-1}}{(\ell-1) !} 
		\left(
		2 \n_\rdr \varphi^{(k-\ell)}
		\right)\\
&=\rho^{-s_0+m}\sum_{\ell = 0}^{k-1} 		
		2\n_\rdr \left (
				\frac{(-\log \rho)^{\ell-1}}{(\ell-1) !} 
				 \varphi^{(k-\ell)}
				\right)
		-2 \left(
			\n_\rdr  \frac{(-\log \rho)^{\ell-1}}{(\ell-1) !}
			\right)
			\varphi^{(k-\ell)} \\
&= 2 \rho^{-s_0+m} \n_\rdr \left( \rho^{s_0-m} \Phi^{(k-1)} \right) + 2 \Phi^{(k-2)} \\
&= 2 \rho^{m} \n_\rdr \left( \rho^{-m} \Phi^{(k-1)} \right) +2 s_0 \Phi^{(k-1)} + 2 \Phi^{(k-2)}. \\
\end{align}
Treating the third part with \eqref{eqn:lapfunctions}
\begin{align}
&\rho^{-s_0+m}\sum_{\ell = 0}^{k-1} 	\frac{1}{\ell !}		
	\left(\Delta (-\log \rho)^\ell \right) \varphi^{(k-\ell)} \\
&=\rho^{-s_0+m}\sum_{\ell = 0}^{k-1} 	\frac{1}{\ell !} 
	\left(
		-\ell(\ell-1) (-\log \rho)^{\ell-2} + (\tr_h B-n)\ell (-\log \rho)^{\ell-1} 
	\right) \varphi^{(k-\ell)} \\
&=   (\tr_h B-n) \Phi^{(k-1)} -\Phi^{(k-2)}.
\end{align}
Combining these calculations provides
\begin{align}\label{eqn:PPhim}
\rho^{-s_0+m} \AVasyDown_{s_0} \rho^{s_0-m} \Phi^{(k)} = (\tr_h B+ 2\rho^m\n_{\rdr}\rho^{-m}) \Phi^{(k-1)}.
\end{align}
The second tedious step in calculating
$\rho^2 \PVasyDown_{s_0-\frac n2} \Phi^{(k)}$
is (recall
$\varphi^{(k-\ell)}\in\ker\div$)
\begin{align}
&\rho^{-s_0+m} (-2\div) \rho^{s_0+m}\Phi^{(k)} \\
& = \rho^{-s_0+m}   \sum_{\ell = 0}^{k-1} 	\frac{2}{\ell !}
	\tr_g \left ( \n (-\log \rho)^\ell \otimes \varphi^{(k-\ell)} \right) \\
& = \rho^{-s_0+m}   \sum_{\ell = 0}^{k-1} 	\frac{(-\log \rho)^{\ell-1}}{(\ell-1) !}
	\left ( - 2\drr \ip  \varphi^{(k-\ell)} \right) \\
&= -2 \drr \ip \Phi^{(k-1)} 
\end{align}
Combing the two previous calculations provides
\begin{align}
\rho^2 \PVasyDown_{s_0-\frac n2} \Phi^{(k)} 
=  \twomatrix{1}{\drr\sp }{0}{1} 
\twovector{2 \rho^m \n_\rdr \rho^{-m} + \tr_h B}{-2 \drr\ip }
\Phi^{(k-1)}
\end{align}
which may be developed upon analysing the following term
\begin{align}
\left (\rho^m \n_\rdr \rho^{-m} - \drr \sp \drr \ip \right ) \Phi^{(k-1)}.
\end{align}
Writing
\begin{align}
\Phi^{(k-1)} = \sum_{\ell=0}^m \sum_{L\in\mathscr{A}^\ell} \Phi^{(k-1)}_{\ell,L}(\rho d\rho)^{m-\ell} dy^L,
\qquad
 \Phi^{(k-1)}_{\ell,L} \in C^\infty( \Xone),
\end{align}
and remarking $\n_\rdr \rho d\rho=2\rho d\rho$ and $\n_\rdr dy^\ell = (1+B)dy^\ell$ gives
\begin{align}
&\left (\rho^m \n_\rdr \rho^{-m} - \drr \sp \drr \ip \right ) \Phi^{(k-1)} \\
&= \left(
		-m + {\rdr} + 2(m-\ell) + (\ell+B) - (m-\ell)
		\right) \sum_{L\in\mathscr{A}^\ell} \Phi^{(k-1)}_{\ell,L}(\rho d\rho)^{m-\ell} dy^L \\
&= ({\rdr} + B) \Phi^{(k-1)}
\end{align}
where
$\rdr$
is to be interpreted as a Lie derivative.
This finally establishes that
\begin{align}
\rho^2 \PVasyDown_{s_0} \Phi^{(k)} = \twovector{ 2 \rdr + 2B + \tr_h B}{-2\drr \sp } \Phi^{(k-1)}
\end{align}
which by induction on 
$k$ 
produces the desired claim that 
$\Phi^{(k)}\in C^\infty_\even(\overline \Xone ; \Sym^m \T^* \Xone)$. 

We extend 
$\Phi^{(k)}$ 
smoothly onto compactly supported sections over 
$\Xthree$ 
and apply 
$\Resolvent{\PVasyDown,m}{s-\frac n2} $ 
to 
$\PVasyDown_{s-\frac n2} \Phi_s$. 
On 
$ \Xone$,
\begin{align}
\Phi_s 
&= \Resolvent{\PVasyDown,m}{s-\tfrac n2} \PVasyDown_{s-\frac n2} \Phi_s \\
&= \rho^{-s+m} J \Resolvent{\QVasyDown,m}{s-\tfrac n2} \twovector{\AVasyDown_s}{-2\div} \rho^{s-m} \Phi_s
\end{align}
whence upon unpacking the definition of 
$\Phi_s$ 
and the expansion of
$\rho^{s+m}$
in
$s$
about
$s_0$
implies
\begin{align}
\varphi_s + O((s-s_0)^j) = \Resolvent{\QVasyDown,m}{s-\tfrac n2}  (s-s_0)^j  \psi_s
\end{align}
for $\psi$ a holomorphic family taking values in 
$C^\infty_\even(\overline \Xone; \oplus_{k=m-1}^m \, \Sym^m \T^* \Xone)$.
Considering the term at order 
$(s-s_0)^{j-1}$ 
provides that 
$\varphi^{(j)}$
is in the image of 
$\Projector{\QVasyDown,m}{s_0-\frac n2}$.
As $\varphi^{(j)}\in C^\infty(\Xone;\Sym^m_0 \T^* \Xone)\cap\ker\div$ and
\begin{align}
\Im \left( \Projector{\Delta,m}{s_0}  \right)
= \Im \left( \Projector{\QVasyDown,m}{s_0-\frac n2} \right) \cap C^\infty(\Xone;\Sym^m_0 \T^* \Xone)\cap\ker\div
\end{align}
we deduce that
$\varphi^{(j)}$
is in the image of
$\Projector{\Delta,m}{s_0}$.
Therefore 
$\AVasyDown^k_{s_0}\varphi^{(j)}$
is also in said image for 
$k\le j$ whence the definition of 
$\varphi^{(k)}$ 
provides the desired result that 
$\varphi$
is in the image of 
$\Projector{\Delta,m}{s_0}$.
\end{proof}


\section{Boundary Distributions and the Poisson Operator}\label{sec:poisson}

Define 
$\Bd_m(\lambda)$ 
to be the following set of boundary distributions
\begin{align}
			\setright{ 
				\omega \in \Dist(\Sn ; \Sym^m_0 \T^*\Sn)
				}{
				\supp(w)\subset K_\Gamma, \,
				U_\gamma^* \omega (y) = T_\gamma (y)^{-\lambda - m} \omega(y)
				\textrm{ for }\gamma\in \Gamma, y\in \Sn
				}.
\end{align}
Then for
$\lambda_0\in\C$
a resonance, we obtain the following identification using 
\eqref{eq:GEquivarianceOfOmega},
\begin{align}
\pi_\Gamma^* \left( \Resonant{A,m}{1}{\lambda_0} \cap \ker \trace \cap \ker \nm \right)
=
(\Phi_-)^{\lambda_0} \mathcal{Q}_- \left( \Bd_m(\lambda_0) \right).
\end{align}

The Poisson operator
is defined via integration of the fibres of 
$\pi_S: S\Hn\to\Hn$. 
For $u\in \Dist (S\Hn; \otimes^m \E^*)$
we define, for
$x\in\Hn$,
\begin{align}
(\pi_{0*} u )(x) : = \int_{S_x \Hn} u(x,\xi) \,dS(\xi)
\end{align}
where integration of elements of
$\otimes^m \E^*$ 
is performed by embedding them in 
$\otimes^m \T^*\Hn$.
For
$\lambda\in\C$, 
the Poisson operator may be now defined as
\begin{align}
\Poisson_\lambda : \left\{	 \begin{array}{rcl}
			\Dist(\Sn ; \Sym^m_0 \T^* \Sn) & \to & C^\infty(\Hn ; \Sym^m_0 \T^*\Hn) \\
			\omega & \mapsto & \pi_{0*} \left( (\Phi_-)^{\lambda} \mathcal{Q}_- \omega \right)
		\end{array}
		\right.
\end{align}
There is a useful change of variables which allows the integral to be performed on the boundary 
$\Sn$.
Specifically, upon introducing the Poisson kernel, we may write
\begin{align}\label{eqn:alternatePoissondefn}
\Poisson_\lambda \omega (x)
=
\int_{\Sn} P(x,y)^{n+\lambda} 
						\left( \otimes^m {\tau_-}_{(x,\xi_{-})}^* \right) \omega(y)
						\,dS(y)
\end{align}
for
$\xi_-=\xi_-(x,y)$.

\subsection{Asymptotics of the Poisson operator}\label{subsec:AssOfPoisson}
We start by recalling a weak expansion detailed in 
\cite[Lemma 6.8]{dfg}.
For this we appeal to the diffeomorphism $\phi$ detailed in 
\cite[Definition 2.1]{h:quantum}.
That is, take
$\rho$ 
an even boundary defining function, from which
the flow of the gradient
$\grad_{\rho^2 g}(\rho)$ induces a diffeomorphism
$\phi:[0,\ve)\times\Sn\to\overline\Hn$.
By implicitly using 
$\phi$
we identify a neighbourhood of the boundary of
$\overline\Hn$
with
$[0,\ve)_\rho\times \Sn$.
Given
$\Psi \in C^\infty(\Sn ; \Sym^m \T\Sn)$ 
we define for
$\rho$
small
\begin{align}
\psi(\rho,y) : = (\otimes^m {\tau_-}_{(x,\xi_-)})\Psi(y)
\end{align}
for
$x=(\rho,y)$
and
$\xi_-=\xi_-(x,y)$.

\begin{lemma}\label{lem:asymptotics-from-dfg}
Let 
$\omega\in \Dist(\Sn;\Sym^m\T^*\Sn)$
and
$\lambda\in \C\backslash(-\frac n2 - \frac 12 \N_0)$.
For each
$y\in \Sn$,
there exists a neighbourhood 
$U_y\subset \overline\Hn$
of 
$y$
and an even boundary defining function
$\rho$
such that for any
$\Psi\in C^\infty(\Sn ; \Sym^m \T\Sn)$
with support contained in
$U_y\cap\Sn$
and giving
$\psi\in C^\infty( (0,\ve)\times\Sn ; \Sym^m \T\Sn)$
as above,
there exists
$F_\pm \in C^\infty_\even ([0,\ve))$
such that
\begin{align}
\int_{\Sn} \left(
		(\Poisson_\lambda \omega)(\rho,y) , \psi(\rho,y)
		\right)
		dS(y)
=
\begin{cases}
\rho^{-\lambda}F_-(\rho) + \rho^{n+\lambda} F_+(\rho), & \lambda\not\in -\frac n2 + \N ; \\
\rho^{-\lambda}F_-(\rho) + \rho^{n+\lambda} \log (\rho) F_+(\rho), & \lambda \in -\frac n2 +  \N ,
\end{cases}
\end{align}
where
$dS$ is the measure obtained from the metric
$\rho^2 g$
restricted to
$\Sn$.
Moreover,
if 
$\omega$
and
$\Psi$
have disjoint supports,
then the expansion may be written
\begin{align}
\begin{cases}
\rho^{n+\lambda} F_+(\rho), & \lambda\not\in -\frac n2 + \N ; \\
\rho^{n+\lambda} ( \log (\rho) F_+(\rho) + F_+'(\rho)), & \lambda \in -\frac n2 +  \N ,
\end{cases}
\end{align}
for 
$F_+'\in C^\infty_\even([0,\ve))$.
\end{lemma}
\qed
\begin{remark}\label{rem:lem-dfg}
The evenness is a consequence of the even expansions of the Bessel functions appearing in the proof.
The additional conclusion when $\omega$ and $\Psi$ have distinct supports is due to Equation~6.31 in the proof as well as the final equation displayed in the proof. In particular, the differential operators (rather than pseudo-differential operators) which appear do not enlarge the supports of $\omega$ and $\Psi$.
Finally, if $\omega$ and $\Psi$ have supports with non-trivial intersection, then $F_-(0)\neq0$.
\end{remark}

\begin{proposition}\label{prop:RuelleToQuantum}
For 
$\lambda\in \C\backslash (-\frac n2 - \frac 12 \N_0)$,
the pushforward map
$\pi_{0*}:\Dist(S\Xone ; \Sym^m\E^*)\to\Dist(\Xone; \Sym^m \T^*\Xone)$
restricts to a linear isomorphism of complex vector spaces
\begin{align}
\pi_{0*}:
\Resonant{A,m}{j}{\lambda_0} \cap \ker \trace \cap \ker \nm
\to
\Resonant{\Delta,m}{j}{\lambda_0+n}.
\end{align}

\end{proposition}

\begin{proof}

Consider 
$u^{(k)}\in\Resonant{A,m}{k}{\lambda_0} \cap \ker \trace \cap \ker \nm$
for
$1\le k \le j$
such that
$(A+\lambda_0)u^{(k)}=-u^{(k-1)}$
and
$(A+\lambda_0)u^{(1)}=0$.
We may suppose that
$u^{(k)}\neq 0$.
We lift these generalised resonant states to
$\tilde u^{(k)}:=\pi_\Gamma^* u^{(k)}$
whose supports are contained in
$\pi_\Gamma^{-1}(K_+)$.
Define
\begin{align}
\tilde\varphi^{(k)} &: = \pi_{0*} \tilde u^{(k)},
\\
\varphi^{(k)} &: = \pi_{0*} u^{(k)}.
\end{align}

Now $\varphi^{(1)}$ is a quantum resonance.
Indeed, the distribution 
$v^{(1)}:=(\Phi_-)^{-\lambda_0}\tilde u^{(1)}$
is annihilated by 
$A$
(as well as both 
$\trace$
and
$\nm$)
so there exists 
$\omega^{(1)}\in \Bd_m(\lambda_0)$
such that
$\tilde u^{(1)} = (\Phi_-)^{\lambda_0} \mathcal{Q}_- w^{(1)}$.
The properties of the Poisson transformation imply that
$\tilde\varphi^{(1)}=\Poisson_{\lambda_0} \tilde u^{(1)}$ 
is trace-free,
divergence-free and in the kernel of
$(\Delta - s_0(n-s_0)-m)$
for
$s_0:=\lambda_0+n$. 
The same statement is true for 
$\varphi^{(1)}$.
Considering the alternative definition for the Poisson operator
\eqref{eqn:alternatePoissondefn},
as well as the upper half-space model, we recall 
the structure of 
$\otimes^m{\tau_-}^*$
from
\eqref{eqn:evenexpansiontau}
and that
$\rho^{-1} P(x,y)$ 
is smooth except at 
$x=(0,y)$. 
Since 
$\omega^{(1)}$
has support contained in 
$K_\Gamma$
disjoint from
$\Omega_\Gamma$
(and
$\overline\Xone=\Gamma\backslash(\Xone\sqcup\Omega_\Gamma)$)
we conclude that
$\varphi^{(1)}\in\rho^{s_0-m} C^\infty_\even(\Xone;\Sym^m\T^*\Xone)$.
This is the characterisation of quantum resonances given in 
Lemma~\ref{lem:AssOfQRes}.
Therefore, as claimed, 
$\varphi^{(1)}$
is a quantum resonance.

We now show that 
$\varphi^{(k)}$
is a generalised quantum resonant.
Define
\begin{align}
v^{(k)} := (\Phi_-)^{-\lambda_0} \sum_{\ell =1}^k \frac{(-\log \Phi_-)^{k-\ell}}{(k-\ell)!} \tilde u^{(\ell)}.
\end{align}
Then a direct calculation shows
$A \tilde v^{(k)}=0$
and, since
$\dm \Phi_-=0$,
it also follows that
$\nm \tilde v^{(k)}=0$.
So let
$\omega^{(k)} \in \Dist(\Sn ; \Sym^m_0 \T^*\Sn)$
with
$\mathcal{Q}_- \omega^{(k)}:= v^{(k)} $
and note
$\supp(w^{(k)})\subset K_\Gamma$.
Rewriting 
$\tilde u^{(k)}$ 
in terms of 
$\tilde v^{(k)}$,
\begin{align}
\tilde u^{(k)} = (\Phi_-)^{\lambda_0} \sum_{\ell =1}^k \frac{(\log \Phi_-)^{k-\ell}}{(k-\ell)!} \tilde v^{(\ell)}
\end{align}
and observing that
\begin{align}
\partial_\lambda^{(k-\ell)} \Poisson_{\lambda_0} \omega^{(\ell)}
=
\pi_{0*} \left( (\Phi_-)^{\lambda_0} (\log \Phi_-)^{k-\ell} \mathcal{Q}_- w^{(\ell)} \right)
\end{align}
we obtain
\begin{align}
\tilde\varphi^{(k)}  = \pi_{0*} \tilde u^{(k)} 
=
\sum_{\ell =1}^k \frac{ \partial_\lambda^{(k-\ell)} \Poisson_{\lambda_0} w^{(\ell)} }{(k-\ell)!}.
\end{align}
Taylor expanding
$(\Delta+\lambda(n+\lambda)-m)\Poisson_\lambda(w^{(k-\ell)})=0$
about
$\lambda_0$
implies
\begin{align}
(\Delta+\lambda_0(n+\lambda_0)-m) \frac{ \partial_\lambda^{(\ell)} \Poisson_{\lambda_0} w^{(k-\ell)} }{\ell!}
+(2\lambda_0+n) \frac{\partial_\lambda^{(\ell-1)} \Poisson_{\lambda_0} w^{(k-\ell)} }{(\ell-1)!}
+ \frac{\partial_\lambda^{(\ell-2)} \Poisson_{\lambda_0} w^{(k-\ell)} }{(\ell-2)!}
=
0.
\end{align}
By introducing (again) $s_0:=\lambda_0+n$, we deduce that
\begin{align}
(\Delta -s_0(n-s_0)-m)\tilde\varphi^{(k)}
=
-(2s_0-n) \tilde\varphi^{(k-1)}
- \tilde\varphi^{(k-2)}
\end{align}
with the interpretation that
$\tilde\varphi^{(0)}=\tilde\varphi^{(-1)}=0$.
By injectivity of the Poisson operator, 
$\varphi^{(k)}\neq 0$.
A similar expansion for
$\div \Poisson_\lambda(w^{(k-\ell)})=0$
implies
$\div \tilde\varphi^{(k)}=0$.
Recalling the definition of the Poisson operator involving the Poisson kernel, we have
$\partial_\lambda^k P(x,y)^{n+\lambda_0}=P(x,y)^{s_0}(\log P(x,y))^{k}$
and so,
as with the case of
$\varphi^{(1)}$,
we conclude
\begin{align}
\varphi^{(k)} \in \oplus_{\ell=0}^{k-1} \,\rho^{s_0-m} (\log \rho)^\ell \,C^\infty_\even (\Xone ; \Sym^m_0 \T^*\Xone).
\end{align}
and so it is a generalised quantum resonance 
$\varphi^{(k)}\in \Resonant{\Delta,m}{k}{\lambda_0+n}$
by
Lemma~\ref{lem:AssOfQRes}.

In order to show surjectivity of
$\pi_{0*}$,
consider
$\varphi^{(j)}\in\Resonant{\Delta,m}{j}{s_0}$
for
$s_0:=\lambda_0+n$
and define
$\varphi^{(k)}$
for
$1\le k < j$
by
$\varphi^{(k)}:=\AVasyDown_{s_0}^{j-k} \varphi^{(j)}\in\Resonant{\Delta,m}{k}{\lambda_0+n}$
(recalling the definition
$\AVasyDown_{s} : = ( \Delta - s(n-s) - m)$).
We may assume 
$\varphi^{(1)}\neq0$.
By modifying
$\varphi^{(k)}$
via linear terms in 
$\varphi^{(\ell)}$
with 
$1\le \ell < k$,
we may assume
\begin{align}
(\Delta -s_0(n-s_0)-m)\varphi^{(k)}
=
-(2s_0-n) \varphi^{(k-1)}
- \varphi^{(k-2)}.
\end{align}
We lift these modified states from
$S\Xone$
to
$S\Hn$
defining
$\tilde\varphi^{(k)}:=\pi_\Gamma^*\varphi^{(k)}$
which also satisfy the preceding display.

We now prove by induction on
$1\le k \le j$
that there exist
$\omega^{(k)} \in \Dist(\Sn ; \Sym^m_0 \T^*\Sn)$
with
$\supp(\omega^{(k)})\subset K_\Gamma$
such that
\begin{align}
\tilde\varphi^{(k)} 
=
\sum_{\ell =1}^k \frac{ \partial_\lambda^{(k-\ell)} \Poisson_{\lambda_0} \omega^{(\ell)} }{(k-\ell)!}
\qquad
\mathrm{and}
\qquad
U_\gamma^* \omega^{(k)} =(T_\gamma)^{-\lambda_0-m} \sum_{\ell =1}^k \frac{(-\log T_\gamma)^{k-\ell}}{(k-\ell)!} \omega^{(\ell)}.
\end{align}
For
$k=1$, 
this states that for 
$\varphi^{(1)}\in\Resonant{\Delta,m}{1}{s_0}$,
there exists
$\omega^{(1)}\in\Bd_m(\lambda_0)$
with
$\pi^*_\Gamma \varphi^{(1)}=\Poisson_{\lambda_0} \omega$.
To  demonstrate this statement we remark that 
$\tilde\varphi^{(1)}$
is tempered on 
$\Hn$,
(the proof follows ad verbum 
\cite[Lemma 4.2]{ghw}),
so the surjectivity of the Poisson transform 
\cite[Corollary 7.6]{dfg} 
provides
$\omega^{(1)} \in \Dist(\Sn ; \Sym^m_0 \T^*\Sn)$
such that
$\tilde\varphi^{(1)}=\Poisson_{\lambda_0}\omega^{(1)}$.
The equivariance property demanded of
$\omega^{(1)}$ 
under 
$\Gamma$
is satisfied as
$\tilde \varphi^{(1)}=\pi_\Gamma^*\varphi^{(1)}$. 
It remains to confirm that
$\supp(\omega^{(1)})\subset K_\Gamma$.
By
Lemma~\ref{lem:AssOfQRes},
we have the asymptotics
$\varphi^{(1)} \in \rho^{s_0-m} C^\infty_\even(\Xone;\Sym^m\T^*\Xone)$
and so, by
Remark~\ref{rem:lem-dfg},
it is only possible for the weak expansion of
Lemma~\ref{lem:asymptotics-from-dfg}
to hold for arbitrary 
$\Psi\in C^\infty( \Omega_\Gamma ; \Sym^m \T^*\Sn)$
if
$\supp(\omega^{(1)})\subset K_\Gamma$.

For the general situation 
$k>1$
consider
\begin{align}
\psi^{(k)} := \tilde \varphi^{(k)} - \sum_{\ell=1}^{k-1} \frac{ \partial_\lambda^{(k-\ell)}\Poisson_{\lambda_0} \omega^{(\ell)}}{(k-\ell)!}
\end{align}
which is in the kernel of
$\AVasyDown_{s_0}$
by a direct calculation. This gives, by the usual argument, a
$\omega^{(k)}\in\Dist(\Sn ; \Sym^m_0 \T^*\Sn)$
with
$\supp(\omega^{(k)})\subset K_\Gamma$
such that
$\psi^{(k)} = \Poisson_{\lambda_0} \omega^{(k)}$
and establishes the first desired equation.
Now consider 
$(\gamma^*-1)\psi^{(k)}$.
As
$(\gamma^*-1)\tilde \varphi^{(k)}=0$
and
$\gamma^*\circ \Poisson_\lambda = \Poisson_\lambda\circ((T_\gamma)^{\lambda+m}U_\gamma^*)$,
the induction hypothesis gives
\begin{align}
(\gamma^*-1)\psi^{(k)} = - \Poisson_{\lambda_0} \left( (T_\gamma)^{\lambda_0+m}\sum_{\ell=1}^{k-1} \frac{(\log T_\gamma)^{k-\ell} }{(k-\ell)!} U_\gamma^* \omega^{k-\ell} \right)
\end{align}
alternatively as
$\psi^{(k)} = \Poisson_{\lambda_0} \omega^{(k)}$,
the equivariance of
$\Poisson_\lambda$
implies
\begin{align}
(\gamma^*-1)\psi^{(k)} = \Poisson_{\lambda_0}(((T_\gamma)^{\lambda_0+m}U_\gamma^* -1) \omega^{(k)}).
\end{align}
From these two equations and the injectivity of the Poisson operator, we obtain the desired equivariance property for 
$U_\gamma^* \omega^{(k)}$.

We now may reproduce in reverse the beginning of the injectivity direction of this proof. Consider the following elements of
$\Dist(S\Hn;\Sym^m_0\E^*)$
\begin{align}
v^{(k)} := \mathcal{Q}_- \omega^{(k)}
\qquad
\mathrm{and}
\qquad
\tilde u^{(k)} := (\Phi_-)^{\lambda_0} \sum_{\ell =1}^k \frac{(\log \Phi_-)^{k-\ell}}{(k-\ell)!} \tilde v^{(\ell)}.
\end{align}
Then
$\tilde u^{(k)}$
is annihilated by
$\nm$.
The equivariance property of
$\omega^{(k)}$
implies that
$(A+\lambda_0)\tilde u^{(k)}=-\tilde u^{(k-1)}$,
that
$(A+\lambda_0)\tilde u^{(1)}=0$,
and that
$\gamma^* \tilde u^{(k)}=\tilde u^{(k)}$.
So these distributions project down giving
$u^{(k)}\in\Dist(S\Xone;\Sym^m_0\E^*)$.
By the support properties of 
$\omega^{(k)}$,
the support of
$ u^{(k)}$ 
is contained in
$K_+$.
Finally, elliptic regularity implies that the wave front sets of 
$u^{(k)}$
are contained in the annihilators of both
$E^n$
and
$E^u$
hence in
$\left. E^{*u} \right|_{K_+}=E_+^*$.
This is the characterisation of Ruelle resonances so the equality
$\pi_{0*}u^{(k)}=\varphi^{(k)}$
implies surjectivity of the pushforward map $\pi_{0*}$.
\end{proof}


\section{Proof of Theorem~\ref{thm:CQCorrespondence}}\label{sec:proof}

We now prove 
Theorem~\ref{thm:CQCorrespondence}.
The following proof in fact gives a more precise statement than that announced in the theorem.
In particular, it shows that the isomorphism respects the Jordan order of generalised resonant states.
\begin{proof}[Proof of Theorem~\ref{thm:CQCorrespondence}.]
Generalised Ruelle resonant states are filtered by Jordan order
\begin{align}
\Resonant{A}{}{\lambda_0}
&=
	\bigoplus_{j=1}^{J(\lambda_0)}		\left( \Resonant{A,0}{j}{\lambda_0} / \Resonant{A,0}{j-1}{\lambda_0} \right)
	\\
&= \bigcup_{j=1}^{J(\lambda_0)}		 \Resonant{A,0}{j}{\lambda_0} .
\end{align}
Restricting to a particular Jordan order 
$j$,
generalised Ruelle resonant states
are filtered into bands via 
\eqref{eq:FilterRuelleIntoBands}
\begin{align}
\Resonant{A,0}{j}{\lambda_0} 
&= 
\bigoplus_{m\in \N_0}				\left( \V{A,m}{j}{\lambda_0} / \V{A,m-1}{j}{\lambda_0} \right).
\end{align}
Each band 
$m$ 
of Jordan order 
$j$ 
is identified via 
Proposition~\ref{prop:IdentifyBandsWithSymmetricPowerOfE} 
(and Proposition~\ref{prop:lambda=m-case})
with vector-valued generalised resonant states for the geodesic flow which are in the kernel of the unstable horosphere operator.
\begin{align}
(\dm)^m:
\V{A,m}{j}{\lambda_0} / \V{A,m-1}{j}{\lambda_0} 
\to
\Resonant{A,m}{j}{\lambda_0+m} \cap \ker \nm.
\end{align}
These generalised resonant states are decomposed via
\eqref{eq:DecomposeMTensorsTrace}
according to their trace
\begin{align}
\Resonant{A,m}{j}{\lambda_0+m} \cap \ker \nm
=
	\bigoplus_{k=0}^{\lfloor \frac m2 \rfloor} \adjtrace^k 
		\left( \Resonant{A,m-2k}{j}{\lambda_0+m} 
			\cap \ker \trace \cap \ker \nm
		\right).
\end{align}
Generalised resonant states of the geodesic flow which are in the kernels of the unstable horosphere operator and the trace operator are identified via 
Proposition~\ref{prop:RuelleToQuantum}
with generalised resonant states of the Laplacian acting on symmetric tensors
\begin{align}
\pi_{0*}:
\Resonant{-X,m-2k}{j}{\lambda_0+m} \cap \ker \trace \cap \ker \nm
&\to
\Resonant{\Delta,m-2k}{j}{\lambda_0+m+n}.\qedhere
\end{align}
\end{proof}


\bibliographystyle{alpha}

\begin{thebibliography}{{Had}16}

\bibitem[BL07]{butterley-liverani}
Oliver Butterley and Carlangelo Liverani.
\newblock Smooth {A}nosov flows: correlation spectra and stability.
\newblock {\em J. Mod. Dyn.}, 1(2):301--322, 2007.

\bibitem[DFG15]{dfg}
Semyon Dyatlov, Fr\'ed\'eric Faure, and Colin Guillarmou.
\newblock Power spectrum of the geodesic flow on hyperbolic manifolds.
\newblock {\em Anal. PDE}, 8(4):923--1000, 2015.

\bibitem[DG16]{dyatlov-g}
Semyon Dyatlov and Colin Guillarmou.
\newblock Pollicott-{R}uelle resonances for open systems.
\newblock {\em Ann. Henri Poincar\'e}, 17(11):3089--3146, 2016.

\bibitem[DS10]{dairbekov-sharafutdinov}
N.~S. Dairbekov and V.~A. Sharafutdinov.
\newblock Conformal {K}illing symmetric tensor fields on {R}iemannian
  manifolds.
\newblock {\em Mat. Tr.}, 13(1):85--145, 2010.

\bibitem[DZ16]{dyatlov-zworski}
Semyon Dyatlov and Maciej Zworski.
\newblock Dynamical zeta functions for {A}nosov flows via microlocal analysis.
\newblock {\em Ann. Sci. \'Ec. Norm. Sup\'er. (4)}, 49(3):543--577, 2016.

\bibitem[FS11]{faure-sjostrand}
Fr\'ed\'eric Faure and Johannes Sj\"ostrand.
\newblock Upper bound on the density of {R}uelle resonances for {A}nosov flows.
\newblock {\em Comm. Math. Phys.}, 308(2):325--364, 2011.

\bibitem[FT13]{faure-tsujii}
Fr\'ed\'eric Faure and Masato Tsujii.
\newblock Band structure of the {R}uelle spectrum of contact {A}nosov flows.
\newblock {\em C. R. Math. Acad. Sci. Paris}, 351(9-10):385--391, 2013.

\bibitem[GHW16]{ghw}
C.~{Guillarmou}, J.~{Hilgert}, and T.~{Weich}.
\newblock {Classical and quantum resonances for hyperbolic surfaces}.
\newblock {\em \arXiv{1605.08801}, to appear in Math. Ann.}, May 2016.

\bibitem[GLP13]{giulietti-liverani-pollicott}
P.~Giulietti, C.~Liverani, and M.~Pollicott.
\newblock Anosov flows and dynamical zeta functions.
\newblock {\em Ann. of Math. (2)}, 178(2):687--773, 2013.

\bibitem[GMP10]{g-moroianu-park}
Colin Guillarmou, Sergiu Moroianu, and Jinsung Park.
\newblock Eta invariant and {S}elberg zeta function of odd type over convex
  co-compact hyperbolic manifolds.
\newblock {\em Adv. Math.}, 225(5):2464--2516, 2010.

\bibitem[Gui05]{g:duke}
Colin Guillarmou.
\newblock Meromorphic properties of the resolvent on asymptotically hyperbolic
  manifolds.
\newblock {\em Duke Math. J.}, 129(1):1--37, 2005.

\bibitem[GZ95]{guillope-zworski:pb}
Laurent Guillop\'e and Maciej Zworski.
\newblock Polynomial bounds on the number of resonances for some complete
  spaces of constant negative curvature near infinity.
\newblock {\em Asymptotic Anal.}, 11(1):1--22, 1995.

\bibitem[{Had}16]{h:quantum}
C.~{Hadfield}.
\newblock {Resonances for Symmetric Tensors on Asymptotically Hyperbolic
  Spaces}.
\newblock {\em \arXiv{1609.06527}, to appear in Anal. PDE}, September 2016.

\bibitem[HMS16]{hms}
Konstantin Heil, Andrei Moroianu, and Uwe Semmelmann.
\newblock Killing and conformal {K}illing tensors.
\newblock {\em J. Geom. Phys.}, 106:383--400, 2016.

\bibitem[Mel93]{melrose:aps}
Richard~B. Melrose.
\newblock {\em The {A}tiyah-{P}atodi-{S}inger index theorem}, volume~4 of {\em
  Research Notes in Mathematics}.
\newblock A K Peters, Ltd., Wellesley, MA, 1993.

\bibitem[MM87]{mazzeo-melrose}
Rafe~R. Mazzeo and Richard~B. Melrose.
\newblock Meromorphic extension of the resolvent on complete spaces with
  asymptotically constant negative curvature.
\newblock {\em J. Funct. Anal.}, 75(2):260--310, 1987.

\bibitem[PP01]{patterson-perry}
S.~J. Patterson and Peter~A. Perry.
\newblock The divisor of {S}elberg's zeta function for {K}leinian groups.
\newblock {\em Duke Math. J.}, 106(2):321--390, 2001.
\newblock Appendix A by Charles Epstein.

\bibitem[Sel56]{selberg}
A.~Selberg.
\newblock Harmonic analysis and discontinuous groups in weakly symmetric
  {R}iemannian spaces with applications to {D}irichlet series.
\newblock {\em J. Indian Math. Soc. (N.S.)}, 20:47--87, 1956.

\bibitem[Vas13a]{v:ml:inventiones}
Andr\'as Vasy.
\newblock Microlocal analysis of asymptotically hyperbolic and {K}err-de
  {S}itter spaces (with an appendix by {S}emyon {D}yatlov).
\newblock {\em Invent. Math.}, 194(2):381--513, 2013.

\bibitem[Vas13b]{v:ml:functions}
Andr\'as Vasy.
\newblock Microlocal analysis of asymptotically hyperbolic spaces and
  high-energy resolvent estimates.
\newblock In {\em Inverse problems and applications: inside out. {II}},
  volume~60 of {\em Math. Sci. Res. Inst. Publ.}, pages 487--528. Cambridge
  Univ. Press, Cambridge, 2013.

\bibitem[Vas17]{v:ml:forms}
Andr\'as Vasy.
\newblock Analytic continuation and high energy estimates for the resolvent of
  the {L}aplacian on forms on asymptotically hyperbolic spaces.
\newblock {\em Adv. Math.}, 306:1019--1045, 2017.

\bibitem[Zwo16]{zworski:vm}
Maciej Zworski.
\newblock Resonances for asymptotically hyperbolic manifolds: {V}asy's method
  revisited.
\newblock {\em J. Spectr. Theory}, 6(4):1087--1114, 2016.

\end{thebibliography}
\def\arXiv#1{\href{http://arxiv.org/abs/#1}{arXiv:#1}}

\end{document}